\documentclass[reqno, 11pt, a4paper]{amsart} 
\usepackage[english]{babel}
\usepackage{amsfonts, amsmath, amsthm, amssymb,amscd,indentfirst}
\usepackage{mathtools}
\usepackage{xcolor}
\colorlet{RED}{red}
\usepackage{newtxtext,newtxmath}
\usepackage{enumerate}
\usepackage{enumitem}
\usepackage{esint}
\usepackage{stackrel}
\usepackage{float}

\usepackage{anysize} 
 \usepackage[mathscr]{euscript}
\marginsize{2cm}{2cm}{2cm}{2cm}

\usepackage{etoolbox}

\makeatletter

\patchcmd{\@tocline}
  {\hfil}
  {\leaders\hbox to 0.6em{\hss.\hss}\hfill}
  {}{}

\renewcommand{\@pnumwidth}{1.5em}

\def\l@section{\@tocline{1}{0.6em}{0em}{}{}}         
\def\l@subsection{\@tocline{2}{0.3em}{2em}{}{}}      
\def\l@subsubsection{\@tocline{3}{0.2em}{4.7em}{}{}}   

\patchcmd{\@tocline}
  {\vskip #1}
  {\vskip #1\relax}
  {}{}

\makeatother

\newtheorem{theorem}{Theorem}[section]
\newtheorem{proposition}[theorem]{Proposition}
\newtheorem{lemma}[theorem]{Lemma}
\newtheorem{definition}[theorem]{Definition}
\newtheorem{corollary}[theorem]{Corollary}
\newtheorem{example}[theorem]{Example}
\newtheorem{remark}[theorem]{Remark}
\newtheorem{cl}{Claim}

 \numberwithin{equation}{section}

  \newcommand{\dvg}{\textnormal{d}v_g}
 \newcommand{\dx}{\textnormal{d}x}
 \newcommand{\ds}{\textnormal{d}\sigma}
  \newcommand{\dsg}{\textnormal{d}\sigma _g}
 
 \newcommand{\gij}{\textnormal{a}  \textnormal{b}}

 \newcommand{\iwe}{ \tau}
 \newcommand{\ops}{\mathcal{Q}}
 \newcommand{\xv}{2_{\mathtt{V}}}
 \newcommand{\xw}{2_{\mathtt{W}}}

 \newcommand{\ck}{\textnormal K}

 \newcommand{\fc}{\mathsf c}
\newcommand{\omb}{\Omega}

\newcommand{\womd}{W^{1,2} (\Omega ; \mathtt{V} _0 , \mathtt{V} _1)}

 \newcommand{\spt}{\operatorname{supp}}

\newcommand{\wghtv}{\mathtt{V}}
\newcommand{\wghtw}{\mathtt{W}}
\newcommand{\dwv}{\textnormal{d} \mathtt{V}}
\newcommand{\dww}{\textnormal{d} \mathtt{W}}
\newcommand{\loc}{\operatorname{loc}}
\newcommand{\kp}{\mathscr{R}}
\newcommand{\rd}{\textnormal{R}}

\newcommand{\mec}{\delta_{\mathbb{R}^n}}

\newcommand{\pvr}{p}
\newcommand{\qxp}{q}
\newcommand{\dy}{\textnormal{d}y}
\newcommand{\cen}{\mathtt{C}}
\newcommand{\bnorm}{\boldsymbol{\|}}

\newcommand{\R}{\mathbb{R}}
\newcommand{\N}{\mathbb{N}}

\newcommand{\qxpz}{q_0}
\newcommand{\qxpo}{q_1}
\newcommand{\Bint}{\mathcal{B}_{\qxp,\qxpz}^{m}}
\newcommand{\Bbd}{\mathcal{B}_{\Gamma,\qxp,\qxpo}^{m}}

\title{{Weighted Sobolev Spaces and an Elliptic Eigenvalue Problem
with an Indefinite Weight on Domains in Noncompact Riemannian Manifolds}}

\author{Juan Pablo Alcon Apaza}

\address{Juan Pablo Alcon Apaza, Departamento de Matem\'atica, Universidade Federal de Minas Gerais, 31270-901, Belo Horizonte - MG,  Brazil}

\email{juanpabloalconapaza@gmail.com}

\begin{document}
\maketitle
\begin{abstract} The aim of this work is to establish sufficient conditions ensuring
the continuity and compactness of the weighted Sobolev embedding and trace operators $$ W^{1,\qxp}(\Omega; \wghtv_0, \wghtv_1) \rightarrow L^{\qxp_0}(\Omega; \wghtv_2) \quad \text { and } \quad W^{1,\qxp}(\Omega; \wghtv_0, \wghtv_1) \rightarrow L^{\qxp_1}(\partial \Omega; \wghtw).$$ As an application, we study the Neumann eigenvalue problem with an indefinite weight \begin{equation*}\left\{\begin{aligned}-\operatorname{div} ( \wghtv _1 \nabla u )+ \wghtv _0 u   & = \lambda \wghtv _2 \iwe u  & & \text { in } \Omega,\\\wghtv _1 \frac{\partial u}{\partial \nu} &= 0 & & \text { on } \partial \Omega .\end{aligned}\right.\end{equation*} where $\Omega$ is an open subset of a noncompact Riemannian manifold, $\lambda$ is a real number, $\iwe$ is a sign-changing function, and $\wghtv _0$, $\wghtv _1$, \(\wghtv_2\), and \(\wghtw\) are weight functions satisfying suitable conditions.  We prove that this problem has infinitely many positive and negative eigenvalues. We also establish boundedness of its weak eigenfunctions and, for $u\in W^{1,2}_0(\overline{\Omega};\wghtv_0,\wghtv_1)$, the decay property  $$\lim_{m\rightarrow\infty}\operatorname*{ess\,sup}_{\Omega\setminus\overline{D_m}}|u|=0,$$ where $(D_m)_{m\in\mathbb N}$ is an exhaustion of the manifold by bounded open sets.\end{abstract}

\let\thefootnote\relax
\let\thefootnote\relax


\markright{{WEIGHTED SOBOLEV SPACES AND AN INDEFINITE ELLIPTIC EIGENVALUE PROBLEM}}

\tableofcontents
\section{Introduction}

Weighted Sobolev spaces are widely used as solution spaces for degenerate elliptic equations. Kufner \cite{kugner1985weightedsobo}, Triebel \cite{triebel1995interpolati}, and Schmeisser \& Triebel \cite{schmeisser1987topics}, among others, have contributed to their study on unbounded domains. See, for example, \cite{choquet2008einsteinequations, chua1992extensiontheorems, gruka1991weighteiii} for the Euclidean case, \cite{hamann2013singularmanifolds, pacini2013conicalsing} for noncompact Riemannian domains, and \cite{zbMATH06237307, post2023heatkernel} for unweighted Sobolev spaces on noncompact Riemannian manifolds.

In the study of partial differential equations on a domain \(\Sigma\subset\mathbb R^n\), it is often useful to know that the embedding of \(W^{m,\qxp}(\Sigma)\) into \(L^{\qxp_0}(\Sigma)\) is compact. For nonlinear variational equations, this property is often used to prove that the associated energy functional satisfies the Palais--Smale condition. When \(\Sigma\) is unbounded, compactness generally fails. Sufficient conditions involving weight functions can be found, for example, in \cite{adams1971compactembedding, gruka1991weighteiii, kufner1990hardytype}.

{ 
The elliptic application considered in this paper is the homogeneous indefinite eigenvalue problem
}
\begin{equation}
\left\{
\begin{aligned}
-\operatorname{div} ( \wghtv _1 \nabla u )+ \wghtv _0 u   & = \lambda \wghtv _2 \iwe u & & \text { in } \Omega,\\
\wghtv _1 \frac{\partial u}{\partial \nu} &= 0 & & \text { on } \partial \Omega .
\end{aligned}
\right. \label{61}
\end{equation}
Here, $\Omega$ is an open subset of a Riemannian manifold $M$ of dimension $n\geq2$, $\nu$ is the outward unit normal vector to $\partial\Omega$, $\lambda$ is a real number, $\iwe:M\rightarrow\mathbb R$ is a sign-changing function, and $\wghtv_0$, $\wghtv_1$, and $\wghtv_2$ are weight functions satisfying suitable conditions.

{ 
We also determine conditions under which weak solutions of \eqref{61} are bounded and satisfy
\begin{equation} \label{87}
\lim_{m\rightarrow\infty}\stackbin[\Omega\setminus\overline{D_m}]{}{\operatorname{ess}\sup}\,|u|=0.
\end{equation}
}
Here, $(D_m)_{m\in\mathbb N}$ is an increasing sequence of bounded open subsets of $M$ such that
$$
M=\bigcup_m D_m, \qquad D_m\Subset D_{m+1},
$$
and $D_m\cap\Omega$ has Lipschitz boundary for every $m\in\mathbb N$. We also define
$$
D^m:=M\setminus\overline{D_m}.
$$

\subsection{Context and Related Work}

{On a closed Riemannian manifold, the Laplace--Beltrami operator has discrete spectrum. For an open subset of $\mathbb R^n$, discreteness depends on the chosen realization of the Laplacian and on the boundary condition. In particular, the Neumann Laplacian has discrete spectrum whenever the embedding $W^{1,2}(\mathcal M)\rightarrow L^2(\mathcal M)$ is compact.}
In general, the spectrum of $\Delta_{\mathcal M}$ need not be discrete. Several nonstandard geometric settings have therefore been studied. Conditions for discreteness of the spectrum of the Laplacian on noncompact complete Riemannian manifolds with specific geometric structures appear in \cite{baider1979discretespectra, brooks1984fiitevolume, bruning1989discretespectrum, donnelly1979purepointspectrum, escobar1986spectrum, kleine1988discretenessconditions}.

Cianchi \& Maz'ya \cite{cianchi2011spectrumnoncomriem} prove that, for a noncompact Riemannian manifold $\mathcal M^n$ with $n\geq2$ and $\mathcal H^n(\mathcal M)<\infty$, the embedding $W^{1,2}(\mathcal M)\rightarrow L^2(\mathcal M)$ is compact if and only if
$$
\lim_{s\rightarrow0}\frac{s}{\mu_{\mathcal M}(s)}=0,
$$
which is equivalent to discreteness of the spectrum of $\Delta_{\mathcal M}$. Here, the isocapacitary function $\mu_{\mathcal M}:[0,\mathcal H^n(\mathcal M)/2]\rightarrow[0,\infty]$ is defined by
\begin{align*}
\mu_{\mathcal M}(s):=\inf\Big\{C(E,G)\:|\: E \text { and } G & \text { are measurable subsets of } \mathcal M,\\
& E\subset G\subset\mathcal M,\quad s\leq\mathcal H^n(E),\quad \mathcal H^n(G)\leq\frac12\mathcal H^n(\mathcal M)\Big\},
\end{align*}
where
\begin{align*}
  C(E, G):=\inf \Bigg\{  \int_{\mathcal M}|\nabla u|^2 \, \dx \:|\: u \in W^{1,2}(M), & u \geq 1  \text  { in } E \text  {  and }  u \leq 0\\
  & \text { in }  \mathcal M \setminus G \text  { (up to a set of standard capacity zero) }\Bigg\}.
  \end{align*}

Furthermore, Cianchi \& Maz'ya \cite{cianchi2013boundseigenfunctions} prove that, if $\mathcal H^n(\mathcal M)<\infty$ and
$$
\int_0\frac{\textnormal{d}s}{\mu_{\mathcal M}(s)}<\infty,
$$
then, for every eigenvalue $\gamma$ of $\Delta_{\mathcal M}$, there exists a constant $C=C(\mu_{\mathcal M},\gamma)$ such that
$$
\|u\|_{L^\infty(\mathcal M)}\leq C\|u\|_{L^2(\mathcal M)}
$$
for every eigenfunction $u$ associated with $\gamma$.

Skrzypczak \& Tintarev \cite{zbMATH07331136} study compact subsets of Sobolev spaces on complete, noncompact, connected Riemannian manifolds. They define orbital discretizations and functions that are quasisymmetric relative to an orbital discretization, and use the spotlight lemma to prove that a bounded set
\[
K\subset H^{1,p}(M)\cap S_{\Gamma,i,\lambda}(M)
\]
is relatively compact in \(L^q(M)\). For a compact, connected group \(G\) of isometries, they prove that if \(G\) is coercive, then \(H_G^{1,p}(M)\) is compactly embedded into \(L^q(M)\). Conversely, compactness implies that \(G\) is coercive provided that the injectivity radius of \(M\) is positive.



{ 
Let $\mathcal M_0$ be a complete Riemannian manifold, and let $\mathsf m\geq0$ be a locally integrable potential. Ouhabaz \cite{ouhabaz2001spectral} studies the positivity of the $L^2$ spectral lower bound of the Schrödinger operator $-\Delta_{\mathcal M_0}+\mathsf m$.
}
He proves that, if $\mathcal M_0$ satisfies local $L^2$-Poincaré inequalities, namely,
$$
\int_{B(\pvr,r)}\left|u-\overline{u_{B(\pvr,r)}}\right|^2\,\dvg
\leq C(R)r^2\int_{B(\pvr,r)}|\nabla u|^2\,\dvg
$$
for every $u\in C^\infty(B(\pvr,r))$, $\pvr\in\mathcal M_0$, and $0<r\leq R$, and if $\mathcal M_0$ satisfies a local doubling property, then
$$
\inf\sigma(-\Delta_{\mathcal M_0}+\mathsf m)>0,
$$
provided that
$$
\inf_{\pvr\in\mathcal M_0}\frac{1}{|B(\pvr,r)|}\int_{B(\pvr,r)}\mathsf m\,\dvg>0
$$
for some $r>0$. Ouhabaz also shows that this mean condition is necessary under additional geometric assumptions on $\mathcal M_0$.

Berestycki \& Rossi \cite{zbMATH06441399} study three definitions of the generalized principal eigenvalue for linear second-order elliptic operators on unbounded domains. They show that several classical properties of the Dirichlet principal eigenvalue may fail in this setting. They also relate these generalized eigenvalues to the maximum principle and to the existence of positive eigenfunctions satisfying Dirichlet boundary conditions. Since the Dirichlet resolvent is not compact in general, they do not rely on the Krein--Rutman theorem. Instead, they use positive sub- and supersolutions, exhaustion by bounded domains, comparison principles, Harnack inequalities, and local elliptic estimates.

If $\Sigma_0$ is a bounded domain in $\mathbb R^n$ with smooth boundary, the eigenvalue problem
$$
\left\{
\begin{aligned}
-\Delta u &=\lambda u & & \text { in } \Sigma_0,\\
u &=0 & & \text { on } \partial\Sigma_0,
\end{aligned}
\right.
$$
has an infinite sequence of positive eigenvalues
$$
0<\lambda_1<\lambda_2\leq\cdots\leq\lambda_k\leq\cdots,\qquad \lambda_k\rightarrow\infty \quad\text { as } k\rightarrow\infty,
$$
each of finite multiplicity. The same properties hold for
$$
\left\{
\begin{aligned}
-\Delta u+a_0(x)u &=\lambda m(x)u & & \text { in } \Sigma_0,\\
u &=0 & & \text { on } \partial\Sigma_0,
\end{aligned}
\right.
$$
when $a_0$ and $m$ are positive and sufficiently regular on $\overline{\Sigma_0}$. For a detailed study of principal eigenvalues for second-order differential operators that are not necessarily in divergence form, see Fleckinger, Hernández, \& Thélin \cite{fleckinger2004existence}.

The existence of principal eigenvalues with positive eigenfunctions is important in nonlinear problems for which positive solutions are sought. This occurs, for example, in reaction--diffusion systems arising in population dynamics, chemical reactions, and combustion; see \cite{smoller2012shock}.

The classical reference for this theory is the book by Courant \& Hilbert \cite{courant1962methods}. The main tool is the spectral theory of compact self-adjoint operators on Hilbert spaces. { The variational characterization yields monotonicity with respect to the coefficients and, under domain inclusion, with respect to the domain. Continuity with respect to coefficients or domains requires an appropriate notion of convergence.}

{ 
The theory also extends to unbounded coefficients and indefinite weights. Assume that $a_0>0$ in $\Sigma_0$, $a_0,m\in L^r(\Sigma_0)$ for some $r>n/2$, and that the positive part of $m$ is nontrivial. Then there exists a unique positive principal eigenvalue $\lambda_1^+>0$ with a positive eigenfunction. If the negative part of $m$ is nontrivial, there is also a unique negative principal eigenvalue $\lambda_1^-<0$ with a positive eigenfunction; see \cite{figuereido1982semilinearelliptic, brown1980existence, weinberger1974variational}.
}



\subsection{Main Results} In this work, we extend the results of Pflüger \cite{pfluger1998compactraces} to noncompact Riemannian manifolds. Under suitable conditions on the weights and on the geometry of \(\Omega\), we establish continuous and compact embeddings and traces for \(W^{1,\qxp}(\Omega;\wghtv_0,\wghtv_1)\); see conditions \ref{5}--\ref{3} below. For other geometric and weighted settings, see, for example, \cite{hamann2013singularmanifolds}, where Sobolev--Slobodeckii and Bessel potential spaces on singular manifolds are studied.

We also obtain { two sequences of eigenvalues} for problem \eqref{61}, following the method of Allegretto \cite{allegreto1992eigenvalueindefiniteweight}. The main difficulty is to prove the decay property \eqref{87} on the unbounded domain \(\Omega\). For Fredholm properties of elliptic operators on noncompact manifolds, see Lockhart \& McOwen \cite{lockartowen1985ellipticoperators} and Lockhart \cite{lockhart1980fredholm}.



In the sequel, we write \(\Gamma=\partial\Omega\). We assume that there exists {a locally finite family of open sets \(U_{k,i},\hat U_{k,i}\subset M\), \((k,i)\in\{0,1\}\times\mathbb N\), covering a neighborhood of \(\overline\Omega\)}, with the following properties:

\begin{enumerate}[label=($U_{\arabic*}$)]
\item \label{5} The family satisfies
\[
\overline{\Omega}
\subset
\left(\bigcup_i U_{0,i}\right)\cup\left(\bigcup_j U_{1,j}\right),
\quad
\bigcup_i\hat U_{0,i}\subset\Omega,
\quad
\Gamma\subset\bigcup_j U_{1,j}, \quad  M= \bigcup _{(k,i)\in \{0,1\} \times \mathbb{N}} \hat U _{k,i},
\]
and \(U_{k,i}\subset\hat U_{k,i}\) for every \((k,i)\in\{0,1\}\times\mathbb N\). 

\item \label{2} For every \((k,i)\in\{0,1\}\times\mathbb N\), there exist \(0<r_{k,i}<\hat r_{k,i}\) and a chart
\[
\psi_{k,i}:B(0,\hat r_{k,i})\subset\mathbb R^n\longrightarrow\hat U_{k,i}\subset M
\]
such that:
\begin{enumerate}[label=($\alph*$)]
   \item The restriction \(\psi_{k,i}:B(0,r_{k,i})\rightarrow U_{k,i}\) is a diffeomorphism.

   \item For every \(j\in\mathbb N\),
   \[
   { 
   \psi_{1,j}\left(B(0,\hat r_{1,j})\cap\{x_n>0\}\right)
   =\hat U_{1,j}\cap\Omega,
   \qquad
   \psi_{1,j}\left(B(0,\hat r_{1,j})\cap\{x_n=0\}\right)
   =\hat U_{1,j}\cap\Gamma.
   }
   \]

   \item For every \(j\in\mathbb N\), the map
   \[
   \psi_{\Gamma,1,j}:B(0,\hat r_{1,j})\cap\{x_n=0\}\longrightarrow\hat U_{1,j}\cap\Gamma,
   \]
   defined by
   \[
   \psi_{\Gamma,1,j}(x_1,\ldots,x_{n-1})
   =\psi_{1,j}(x_1,\ldots,x_{n-1},0),
   \]
   is a chart for \(\Gamma\).
\end{enumerate}

\item \label{1} There exists a constant \(\rd_1>0\) such that
\[
\sum_{k,i}\chi_{\hat U_{k,i}}\leq\rd_1
\quad\text{in }M.
\]

\item \label{3} There exists a constant \(\rd_2>0\) such that, for every \(j\in\mathbb N\),
\[
\left(
\sup_{\substack{x\in B(0,\hat r_{1,j})\\ |z|_{\delta_{\mathbb R^n}}=1}}
\left|d(\psi_{1,j})_x z\right|_g
\right)
\left(
\sup_{\substack{(\pvr,v)\in T\hat U_{1,j}\\ |v|_g=1}}
\left|d(\psi_{1,j}^{-1})_{\pvr}v\right|_{\delta_{\mathbb R^n}}
\right)
\leq\rd_2
\]
and
\[
\left(\sup_{\hat U_{1,j}}\det[g_{\gij}]\right)
\left(\sup_{\hat U_{1,j}}\det[g^{\gij}]\right)
\leq\rd_2^2.
\]
Here, \(\delta_{\mathbb R^n}\) is the Euclidean metric,
\([g^{\gij}]:=[g_{\gij}]^{-1}\), and
\(g_{\gij}=g(d\psi_{1,j}e_{\textnormal a},d\psi_{1,j}e_{\textnormal b})\).
\end{enumerate}

\begin{remark} $ $
\begin{enumerate}[label=(\roman*)] 

\item Examples satisfying conditions \ref{5}--\ref{3} can be constructed in the hyperbolic ball \(\mathbb B^n\). In Euclidean space, examples include cylinders and the half-space.

\item Conditions \ref{5}--\ref{3} extend to Riemannian manifolds the conditions U1--U4 in \cite{pfluger1998compactraces}; see also \cite{zbMATH01109876}. These conditions control the geometry of \(\partial\Omega\) at infinity.

\item \label{109} In conditions \ref{5}--\ref{3}, the balls \(B(0,r_{1,j})\) and \(B(0,\hat r_{1,j})\) may be replaced by the cubes
\[
Q(0,r_{1,j})
=\{x\in\mathbb R^n\mid\|x\|_{\max}<r_{1,j}\}
\]
and
\[
Q(0,\hat r_{1,j})
=\{x\in\mathbb R^n\mid\|x\|_{\max}<\hat r_{1,j}\},
\]
without changing the results.

\item If \((M,\Omega)\) has {\it bounded geometry}; see \cite[Definition 4.1]{zbMATH06237307} and \cite[Proposition 3.2]{zbMATH01606240}, then conditions \ref{5}--\ref{3} hold for a suitable family of {\it geodesic normal coordinates} and {\it Fermi coordinates}.

\end{enumerate}
\end{remark}



Let \(\wghtv_i\), \(i=0,\ldots,3\), and \(\wghtw\) be weights on \(M\); that is, they are locally integrable and positive almost everywhere. We also assume that \(\wghtw\) is continuous and that all the weights are bounded above and below by positive constants on every compact subset of \(M\).

We assume that there exist positive, continuous functions $b_i:M\rightarrow\mathbb R$, $i=1,2,3$, and constants $\ck_0,\ck_{\qxp}>0$ such that the following conditions hold.

If $(k,i)\in\{0,1\}\times\mathbb N$ and $j\in\mathbb N$, then
\begin{enumerate}[label=($\mathtt{W}_{\arabic*}$)]

\setcounter{enumi}{-1}

\item \label{138}
For every $x\in B(0,\hat r_{1,j})\cap\{x_n<0\}$ and every $\ell\in\{0,1\}$,
$$
\wghtv_\ell(\psi_{1,j}(x))
\leq
\ck_0\wghtv_\ell(\psi_{1,j}(x_1,\ldots,x_{n-1},-x_n)).
$$

\item \label{45}
For almost every $\pvr\in\hat U_{k,i}$,
$$
\kp_{k,i}^{\qxp}(\pvr)\wghtv_1(\pvr)
\leq
\ck_{\qxp}\wghtv_0(\pvr).
$$

\item \label{139} For almost every $\pvr\in\hat U_{k,i}$,
$$
\hat r_{k,i}^{-\qxp}\bnorm d\psi_{k,i}\bnorm^{-\qxp}\wghtv_1(\pvr)
\leq
\ck_{\qxp}\wghtv_0(\pvr).
$$

\item \label{46}
For almost every $\pvr\in\hat U_{k,i}$,
$$
\wghtv_2(\pvr)
\leq
b_2(\psi_{k,i}(0))
\quad\text{and}\quad
\bnorm d\psi_{k,i}\bnorm^{\qxp} b_1(\psi_{k,i}(0))
\leq
\wghtv_1(\pvr).
$$

\item \label{16}
For almost every $\pvr\in\hat U_{1,j}$,
$$
\wghtw(\pvr)
\leq
b_3(\psi_{1,j}(0))
\quad\text{and}\quad
{ \bnorm d\psi_{1,j}\bnorm^{\qxp}} b_1(\psi_{1,j}(0))
\leq
\wghtv_1(\pvr).
$$
\end{enumerate}


Here,
\begin{align*}
\kp_{k,i}
&\displaystyle:=
\sum_{(\textnormal a,\textnormal b)\in\{0,1\}\times\mathbb N}
\bnorm d\psi_{\textnormal a,\textnormal b}^{-1}\bnorm
(\hat r_{\textnormal a,\textnormal b}-r_{\textnormal a,\textnormal b})^{-1}
\chi_{\hat U_{\textnormal a,\textnormal b}\cap\hat U_{k,i}},\\[5pt]
\bnorm d\psi_{k,i}\bnorm
&\displaystyle:=
\sup\left\{
|d(\psi_{k,i})_x z|_g
\mid
x\in B(0,\hat r_{k,i}),\ |z|_{\delta_{\mathbb R^n}}=1
\right\},\\[5pt]
\bnorm d\psi_{k,i}^{-1}\bnorm
&\displaystyle:=
\sup\left\{
|d(\psi_{k,i}^{-1})_{\pvr}v|_{\delta_{\mathbb R^n}}
\mid
(\pvr,v)\in T\hat U_{k,i},\ |v|_g=1
\right\}.
\end{align*}

Next, set
\begin{equation*}
\mathcal B_{\qxp,\qxp_0}^m
:=
\sup_{\substack{(k,i)\in\{0,1\}\times\mathbb N\\ \psi_{k,i}(0)\in D^m}}
\frac{b_2^{1/\qxp_0}(\psi_{k,i}(0))}{b_1^{1/\qxp}(\psi_{k,i}(0))}
\bnorm G_{k,i}\bnorm^{1/\qxp_0}
\bnorm G_{k,i}^{-1}\bnorm^{1/\qxp}
\hat r_{k,i}^{\frac n{\qxp_0}-\frac n{\qxp}+1},
\end{equation*}
where
\[
\bnorm G_{k,i}\bnorm
:=\sup_{\hat U_{k,i}}\sqrt{\det[g_{\gij}]}
\quad\text{and}\quad
\bnorm G_{k,i}^{-1}\bnorm
:=\sup_{\hat U_{k,i}}\sqrt{\det[g^{\gij}]}.
\]

Furthermore, define
\begin{equation*}
\mathcal B_{\Gamma,\qxp,\qxp_1}^m
:=
\sup_{\substack{j\in\mathbb N\\ \psi_{1,j}(0)\in\Gamma^m}}
\frac{b_3^{1/\qxp_1}(\psi_{1,j}(0))}{b_1^{1/\qxp}(\psi_{1,j}(0))}
\bnorm G_{\Gamma,1,j}\bnorm^{1/\qxp_1}
\bnorm G_{1,j}^{-1}\bnorm^{1/\qxp}
\hat r_{1,j}^{\frac{n-1}{\qxp_1}-\frac n{\qxp}+1},
\end{equation*}
where
\[
{ 
\bnorm G_{\Gamma,1,j}\bnorm
:=
\sup_{\hat U_{1,j}\cap\Gamma}\sqrt{\det[g_{\Gamma,\gij}]}
}
\quad\text{and}\quad
g_{\Gamma,\gij}
:=g(d\psi_{\Gamma,1,j}e_{\textnormal a},d\psi_{\Gamma,1,j}e_{\textnormal b}).
\]

\begin{remark} $ $
\begin{enumerate}[label=(\roman*)] 

\item The quantities \(\mathcal B_{q,q_0}^m\) and \(\mathcal B_{\Gamma,q,q_1}^m\)  extend to Riemannian manifolds the quantities \(\mathcal B_{n,k}\) in \cite{pfluger1998compactraces} and \(\mathscr A_n\) in \cite{gruka1991weighteiii}, which are defined in the Euclidean setting.


\item If \((M,\Omega)\) has {\it bounded geometry}; see \cite[Definition 4.1]{zbMATH06237307} and \cite[Proposition 3.2]{zbMATH01606240}, then one may choose geodesic normal and Fermi coordinates such that { there exists \(c\geq1\) for which each of
\[
\kp_{k,i},\quad
\bnorm G_{k,i}\bnorm,\quad
\bnorm G_{k,i}^{-1}\bnorm,\quad
\bnorm G_{\Gamma,1,j}\bnorm,\quad
\bnorm d\psi_{k,i}\bnorm,\quad
\bnorm d\psi_{k,i}^{-1}\bnorm,\quad
\hat r_{k,i},\quad
\hat r_{k,i}-r_{k,i}
\]
is bounded above by \(c\) and below by \(c^{-1}\)}.
\end{enumerate}
\end{remark}

\begin{example}\label{ex:shrinking-components-eta}
Let $\eta>0$. Let $\ell_0 > 4^{1/\eta}$. For every
$\ell\in\N$ with $\ell\geq\ell_0$, set
\[
a_\ell:=(\ell,0),
\qquad
R_\ell:=\left( \frac1\ell \right)^\eta,
\qquad
\rho_\ell:=\sqrt{R_\ell}.
\]
Consider the open set
\[
{ 
\Omega
:=
\bigcup_{\substack{\ell\in\N\\ \ell\geq\ell_0}}B(a_\ell,R_\ell)
\subset\R^2.
}
\]
Endow \(\mathbb R^2\) with the Euclidean metric \(g=\delta_{\R^2}\). { For this example, take \(D_m=B(0,m)\).} 

Define
\[
r_{0,\ell}:=\left(1-\frac{R_\ell}{4}\right)\rho_\ell,
\qquad
\hat r_{0,\ell}:=\left(1-\frac{R_\ell}{8}\right)\rho_\ell,
\]
and
\[
r_{1,\ell}:=\frac{3\rho_\ell}{4},
\qquad
\hat r_{1,\ell}:=\frac{7\rho_\ell}{8}.
\]

The interior chart is
\[
\psi_{0,\ell}(x,y):=a_\ell+\rho_\ell(x,y),
\qquad
(x,y)\in B(0,\hat r_{0,\ell}).
\]

For the boundary charts, put
\[
S_\ell(x):=\sqrt{R_\ell^2-\rho_\ell^2x^2}.
\]
Define
\[
\psi_{1,\ell}^{\pm }(x,y)
:=
a_\ell+
\left(1-\frac{y}{\rho_\ell}\right)
\left(
\rho_\ell x,
\pm S_\ell(x)
\right),
\qquad
(x,y)\in B(0,\hat r_{1,\ell}),
\]
and
\[
\varphi_{1,\ell}^{\pm}(x,y)
:=
a_\ell+
\left(1-\frac{y}{\rho_\ell}\right)
\left(
\pm S_\ell(x),
\rho_\ell x
\right),
\qquad
(x,y)\in B(0,\hat r_{1,\ell}).
\]

After relabeling the charts
\[
\psi_{0,\ell},
\qquad
\psi_{1,\ell}^{+},
\quad
\psi_{1,\ell}^{-},
\quad
\varphi_{1,\ell}^{+},
\quad
\varphi_{1,\ell}^{-},
\qquad
\ell\in\N,\ \ell\geq\ell_0,
\]
as a family $\psi_{k,i}$, the corresponding covering satisfies conditions
\ref{5}--\ref{3}. Moreover, there exists a constant $C>0$, independent of
$\ell$, such that, for every $\ell\in\N$ with $\ell\geq\ell_0$,
\begin{gather*}
\bnorm d\psi_{0,\ell}\bnorm, \
\bnorm d\psi_{1,\ell}^{\pm}\bnorm,  \
\bnorm d\varphi_{1,\ell}^{\pm}\bnorm
\leq
C\rho_\ell,
\label{eq:dpsi-shrinking-components-eta}
\\[3pt]
\bnorm d(\psi_{0,\ell})^{-1}\bnorm, \
\bnorm d(\psi_{1,\ell}^{\pm})^{-1}\bnorm, \
\bnorm d(\varphi_{1,\ell}^{\pm})^{-1}\bnorm
\leq
\frac{C}{\rho_\ell},
\label{eq:inverse-dpsi-shrinking-components-eta}
\\[3pt]
\bnorm G_{\psi_{0,\ell}}\bnorm,
\bnorm G_{\psi_{1,\ell}^{\pm}}\bnorm, \
\bnorm G_{\varphi_{1,\ell}^{\pm}}\bnorm
\leq
CR_\ell, 
\qquad
\bnorm G_{\psi_{0,\ell}}^{-1}\bnorm, \
\bnorm G_{\psi_{1,\ell}^{\pm}}^{-1}\bnorm, \
\bnorm G_{\varphi_{1,\ell}^{\pm}}^{-1}\bnorm
\leq
\frac{C}{R_\ell},
\label{eq:G-shrinking-components-eta}
\\[3pt]
\bnorm G_{\Gamma,\psi_{1,\ell}^{\pm}}\bnorm, \
\bnorm G_{\Gamma,\varphi_{1,\ell}^{\pm}}\bnorm
\leq
C\rho_\ell, 
\label{eq:Gamma-G-shrinking-components-eta}\\[3pt]
\kp_{\psi_{0,\ell}}, \
\kp_{\psi_{1,\ell}^{\pm}}, \
\kp_{\varphi_{1,\ell}^{\pm}}
\leq
\frac{C}{R_\ell^2} .
\label{eq:kappa-shrinking-components-eta}
\end{gather*}

Define the weights
\[
\wghtv_0(p)
:=
(1+|p|)^{2\eta\qxp+\alpha_0},
\qquad
\alpha_0\geq0,
\]
\[
\wghtv_1(p):=1,
\]
\[
\wghtv_2(p)
:=
(1+|p|)^{-\alpha_2},
\qquad
\alpha_2\geq0,
\]
and
\[
\wghtw(p)
:=
(1+|p|)^{-\beta},
\qquad
\beta\geq0.
\]

There exist positive continuous functions $b_1$, $b_2$, and $b_3$ such that
\[
b_1(x,y)\approx1,
\qquad
b_2(x,y)\approx(1+|(x,y)|)^{-\alpha_2},
\qquad
b_3(x,y)\approx(1+|(x,y)|)^{-\beta}.
\]

Then the weights satisfy conditions \ref{138}--\ref{16}.

Assume that \(\qxp_0\geq\qxp\geq1\) and
\[
-\frac{\alpha_2}{\qxp_0}
-\frac{2\eta}{\qxp_0}
+\frac{2\eta}{\qxp}
-\frac{\eta}{2}
\leq0.
\]
Then there exist constants \(C_1>0\) and \(m_0\in\N\) such that, for every \(m\in\N\) with \(m\geq m_0\),
\[
\mathcal{B}_{\qxp,\qxp_0}^m
\leq
C_1
m^{
-\frac{\alpha_2}{\qxp_0}
-\frac{2\eta}{\qxp_0}
+\frac{2\eta}{\qxp}
-\frac{\eta}{2}
}.
\]

Similarly, assume that $\qxp_1\geq1$ and
\[
-\frac{\beta}{\qxp_1}
-\frac{\eta}{\qxp_1}
+\frac{2\eta}{\qxp}
-\frac{\eta}{2}
\leq0.
\]
Then there exist constants $C_1>0$ and $m_0\in\N$ such that, for every
$m\in\N$ with $m\geq m_0$,
\[
\mathcal{B}_{\Gamma,\qxp,\qxp_1}^m
\leq
C_1
m^{
-\frac{\beta}{\qxp_1}
-\frac{\eta}{\qxp_1}
+\frac{2\eta}{\qxp}
-\frac{\eta}{2}
}.
\]

\end{example}

\begin{example}\label{ex:polynomial-alpha-horn-domain}
Let \(\alpha>0\) and \(m_0\geq2\). Let \(\omb\subset\R^2\) be a smooth domain, and endow \(\mathbb R^2\) with the Euclidean metric \(g=\delta_{\R^2}\).
Assume that
\[
 \overline{\omb}\cap\{(X,Y)\in\R^{2} \mid X\leq m_{0}\}
\]
is compact and that
\[
 \omb\cap\{X>m_{0}\}
 =
 \left\{
 (X,Y)\in\R^{2} \mid
 X>m_{0},\ |Y|<\frac{1}{X^\alpha}
 \right\}.
\]

Let
\[
 1\leq \qxp,\qxpz,\qxpo<\infty,
 \qquad
 \eta,\beta,\delta,\gamma\in\R.
\]
Consider the weights
\[
 \wghtv_1(p):=(1+|p|)^{1-\beta},
 \qquad
 \wghtv_0(p):=(1+|p|)^{\qxp-\eta},
\]
and
\[
 \wghtv_2(p):=(1+|p|)^\delta,
 \qquad
 \wghtw(p):=(1+|p|)^\gamma.
\]
Assume
\[
 \beta\geq1, \qquad \eta \geq  q,  
 \qquad
 \beta\geq \eta+( \alpha-1)\qxp+1.
\]

Fix
\[
 r\in
 \left(
 \max\left\{\frac12,\frac1{\alpha+1}\right\},
 1
 \right).
\]
Choose \(N\in\N\) such that
\[
N\geq m_0^{\alpha+1}.
\]
{ 
For \(i\in\N\), set
\[
a_i:=(N+i)^{1/(\alpha+1)}.
\]
For this example, take \(D_m=B(0,m)\).
}

Set
\[
r_{0,i}=r_{1,i}:=r,
\qquad
\hat r_{0,i}=\hat r_{1,i}:=1.
\]

For \((s,t)\in Q(0,1)=(-1,1)^2\), define
\[
 X_i(s):=a_i+\frac{s}{a_i^\alpha}.
\]
The interior charts on the tail are
\[
 \varphi_{0,i}(s,t)
 :=
 \left(
 X_i(s),
 \frac{t}{X_i(s)^\alpha}
 \right),
 \qquad i\in\N.
\]
The upper and lower boundary charts on the tail are
\[
 \psi_{1,i}^{\pm}(s,t)
 :=
 \left(
 X_i(s),
 \pm\frac{1-t}{X_i(s)^\alpha}
 \right),
 \qquad i\in\N.
\]

After adjoining finitely many ordinary interior charts and finitely many
ordinary boundary charts on the compact part, and after relabeling the
families
\[
 \varphi_{0,i},
 \qquad
 \psi_{1,i}^{+},
 \qquad
 \psi_{1,i}^{-}
\]
as a family \(\psi_{k,j}\), \((k,j)\in\{0,1\}\times\N\), the resulting
covering satisfies conditions \ref{5}--\ref{3},
with cubes replacing balls as permitted by Remark \ref{109}.

Moreover, there exists a constant \(C>0\), independent of \(i\), such that
\begin{gather*}
 \frac1{Ca_i^\alpha}
 \leq
 \bnorm d\varphi_{0,i}\bnorm, \ \bnorm d\psi_{1,i}^{\pm}\bnorm
 \leq
 \frac{C}{a_i^\alpha}, \qquad
 \frac{a_i^\alpha}{C}
 \leq
 \bnorm d\varphi_{0,i}^{-1}\bnorm, \ \bnorm d(\psi_{1,i}^{\pm})^{-1}\bnorm
 \leq
 Ca_i^\alpha, \\[3pt]
 \frac{1}{Ca_i^{2\alpha}}
 \leq
 \bnorm G_{\varphi_{0,i}}\bnorm , \ \bnorm G_{\psi_{1,i}^{\pm}}\bnorm
 \leq
 \frac{C}{a_i^{2\alpha}},\qquad
 \frac{a_i^{2\alpha}}{C}
 \leq
  \bnorm G_{\varphi_{0,i}}^{-1}\bnorm, \ \bnorm G_{\psi_{1,i}^{\pm}}^{-1}\bnorm
 \leq
 Ca_i^{2\alpha},\\[3pt]
 \frac1{Ca_i^\alpha}
 \leq
 \bnorm G_{\Gamma,\psi_{1,i}^{\pm}}\bnorm
 \leq
 \frac{C}{a_i^\alpha},\\[3pt]
 \kp_{\varphi_{0,i}}
 \leq
 C_r a_i^\alpha,
 \qquad
 \kp_{\psi_{1,i}^{\pm}}
 \leq
 C_r a_i^\alpha.
 \end{gather*}
 
There exist positive continuous functions $b_1$, $b_2$, and $b_3$ such that, on the tail,
\[
b_1(s,t)\approx (1+|(s,t)|)^{1-\beta},
\qquad
b_2(s,t)\approx(1+|(s,t)|)^{\delta},
\qquad
b_3(s,t)\approx(1+|(s,t)|)^{\gamma}.
\]
With this choice, the weights  satisfy conditions \ref{138}--\ref{16}.

If
\[
 \frac{\delta-2\alpha}{\qxpz}
 +
 \frac{\beta+2\alpha-1}{\qxp}
 \leq0,
\]
then there exist constants \(C>0\) and \(m_\ast\geq1\), independent of
\(m\), such that, for every \(m\geq m_\ast\),
\[
 \Bint
 \leq
 C
 m^{
 \frac{\delta-2\alpha}{\qxpz}
 +
 \frac{\beta+2\alpha-1}{\qxp}
 }.
\]
If
\[
 \frac{\gamma-\alpha}{\qxpo}
 +
 \frac{\beta+2\alpha-1}{\qxp}
 \leq0,
\]
then there exist constants \(C>0\) and \(m_\ast\geq1\), independent of
\(m\), such that, for every \(m\geq m_\ast\),
\[
 \Bbd
 \leq
 C
 m^{
 \frac{\gamma-\alpha}{\qxpo}
 +
 \frac{\beta+2\alpha-1}{\qxp}
 }.
\]

\end{example}

\subsubsection{{Analysis of Embeddings}} Our first main results are Propositions \ref{33} and \ref{56}. We write
\[
\Omega_m:=D_m\cap\Omega,
\qquad
\Omega^m:=\Omega\setminus\overline{\Omega_m},
\qquad
\Gamma_m:=D_m\cap\Gamma,
\qquad
\Gamma^m:=\Gamma\setminus\overline{\Gamma_m}.
\]
\begin{proposition} \label{33}
Assume that conditions \ref{5}--\ref{3}, \ref{139},  and \ref{46} hold. Suppose that \(1\leq\qxp<n\) and
$$
\frac{n\qxp }{ n - \qxp } \geq \qxp _0 \geq \qxp.
$$
Then the following statements hold:
\begin{enumerate}[label=(\roman*)]

\item If 
$$
\lim _{m \rightarrow \infty} \mathcal{B} _{ \qxp , \qxp _0 } ^m  <\infty,
$$
then 
$$
W^{1, \qxp  }(\Omega ;  \wghtv _0, \wghtv _1 ) \rightarrow L^{ \qxp _0 }  ( \Omega ; \wghtv _2 )$$ 
is continuous.

\item If  
$$
\qxp _0<\frac{n\qxp}{n-\qxp} \quad \text { and } \quad \lim _{m \rightarrow \infty} \mathcal{B} _{ \qxp , \qxp _0 } ^m =0,
$$ 
then 
$$
W^{1, \qxp  } (\Omega ; \wghtv _0,  \wghtv _1 ) \rightarrow L^{ \qxp _0 }  ( \Omega ; \wghtv _2  )$$
 is compact.
\end{enumerate}
\end{proposition}

\begin{proposition} \label{56}
Assume that conditions \ref{5}--\ref{3}, \ref{139},  and \ref{16} hold. Suppose that \(1\leq\qxp<n\) and
\[
 \frac{(n-1)\qxp }{n - \qxp}  \geq \qxp _1 \geq \qxp.
\]
Then the following statements hold:
 \begin{enumerate}[label=(\roman*)]
\item If  
$$
\lim _{m \rightarrow \infty} \mathcal{B} _{ \Gamma , \qxp , \qxp _1 } ^m <\infty,
$$
 then there exists a continuous trace operator $$
 W^{1, \qxp  } ( \Omega ; \wghtv_0, \wghtv_1 ) \rightarrow L^{\qxp _1}  (\Gamma ; \wghtw).
 $$

\item  If 
$$
\qxp _1<\frac{(n-1)\qxp}{n-\qxp} \quad \text { and } \quad \lim _{m \rightarrow \infty} \mathcal{B} _{ \Gamma , \qxp , \qxp _1 } ^m =0,
$$
 then the trace operator 
 $$
 W^{1, \qxp  } (\Omega ; \wghtv _0, \wghtv _1 ) \rightarrow L^{\qxp _1 }  (\Gamma ; \wghtw)
 $$
  is compact.
\end{enumerate}
\end{proposition}


\begin{corollary}\label{57}

Suppose that \(n\geq3\) and that conditions \ref{5}--\ref{3} and \ref{139}--\ref{16} hold for \(\qxp=2\). Assume further that
\[
\lim_{m\rightarrow\infty}\mathcal B_{2,\xv}^m<\infty
\quad\text{and}\quad
\lim_{m\rightarrow\infty}\mathcal B_{\Gamma , 2,\xw}^m<\infty.
\]
Then there exist constants \(\mathcal C_{\xv}>0\) and \(\mathcal C_{\xw}>0\) such that
\begin{gather}
\|u\|_{\xv,\Omega,\wghtv_2}
\leq
\mathcal C_{\xv}\|u\|_{1,2,\Omega,\wghtv_0,\wghtv_1},
\label{62}\\[5pt]
\|u\|_{\xw,\Gamma,\wghtw}
\leq
\mathcal C_{\xw}\|u\|_{1,2,\Omega,\wghtv_0,\wghtv_1},
\label{63}
\end{gather}
for every \(u\in W^{1,2}(\Omega;\wghtv_0,\wghtv_1)\).
\end{corollary}



\subsubsection{Study of Eigenvalues} We now assume the following condition.
\begin{enumerate}[label=($\mathtt{W}_{\arabic*}$)]
\setcounter{enumi}{3}
\item \label{47} For every \((k,i)\in\{0,1\}\times\mathbb N\) and almost every \(\pvr\in\hat U_{k,i}\),
\[
{
\kp_{k,i}^{2}(\pvr)\wghtv_1(\pvr)
\leq\ck_2\wghtv_2(\pvr)
}
\quad\text{and}\quad
\int_{\hat U_{k,i}}\wghtv_2\,\dvg\leq\ck_3.
\]
\end{enumerate}


For weights \(\wghtv\) and \(\wghtw\), we use the notation
\[
\dwv=\wghtv\,\dvg,
\qquad
\dww=\wghtw\,\dsg,
\qquad
\wghtv(U)=\int_U\wghtv\,\dvg,
\qquad
\wghtw(U\cap\Gamma)=\int_{U\cap\Gamma}\wghtw\,\dsg,
\]
for every measurable set \(U\subset M\).


\begin{theorem}\label{59}
Assume the hypotheses of Corollary \ref{57}. Suppose that \(\wghtv_1\in C^1(M)\),
\[
\iwe\in L^{\xv/(\xv-2)}(\Omega;\wghtv_2)\cap L^\infty(\Omega),
\]
and that the sets
\[
\{\pvr\in\Omega\mid\iwe(\pvr)>0\}
\quad\text{and}\quad
\{\pvr\in\Omega\mid\iwe(\pvr)<0\}
\]
have nonempty interiors. Then problem \eqref{61} has two infinite sequences of eigenvalues
\[
\cdots\leq\lambda_2^-\leq\lambda_1^-<0<\lambda_1^+\leq\lambda_2^+\leq\cdots,
\]
{with \(\lambda_k^-\rightarrow-\infty\) and \(\lambda_k^+\rightarrow+\infty\) as \(k\rightarrow\infty\).}

If \(u\in\womd\) is an eigenfunction of \eqref{61} and condition \ref{47} holds, then \(u\in L^\infty(\Omega)\). Moreover, if \(u\in W_0^{1,2}(\overline\Omega;\wghtv_0,\wghtv_1)\), then \eqref{87} holds.

{Here, \(W_0^{1,2}(\overline\Omega;\wghtv_0,\wghtv_1)\) denotes the closure in \(W^{1,2}(\Omega;\wghtv_0,\wghtv_1)\) of the restrictions to \(\Omega\) of functions in \(C_0^\infty(M)\).}
\end{theorem}


%


\medskip

The remainder of the paper is organized as follows. Section \ref{141} presents the preliminary definitions and results. In Section \ref{142}, we establish an extension result and prove continuous and compact weighted Sobolev embedding and trace theorems. Section \ref{143} is devoted to elliptic estimates and the analysis of problem \eqref{61}. In particular, we prove the existence of infinitely many positive and negative eigenvalues and, under additional assumptions, establish the boundedness and decay at infinity of the corresponding eigenfunctions. Finally, the appendix contains the proof of a technical lemma used to derive the elliptic estimates.

\section{Preliminaries} \label{141}

Let $(M,g)$ be a smooth Riemannian manifold. For every { nonnegative} integer $k$ and every smooth function $u:M\rightarrow\mathbb{R}$, we denote by $\nabla^k u$ the $k$th covariant derivative of $u$. {  We set $\nabla^0u=u$.} For $k\geq1$, the pointwise norm of $\nabla^k u$ is defined in local coordinates by
\[
{ 
\left|\nabla^k u\right|^2
=
g^{i_1j_1}\cdots g^{i_kj_k}
(\nabla^k u)_{i_1\ldots i_k}
(\nabla^k u)_{j_1\ldots j_k}.
}
\]

Recall that $(\nabla u)_i=\partial_i u$, while
\begin{equation}\label{137}
{ 
(\nabla^2u)_{ij}=\partial_{ij}u-\Gamma_{ij}^{\ell}\partial_{\ell}u.
}
\end{equation}

We use the following weighted Sobolev spaces; see \cite{kilpel1994weightedcapacity, turesson2000potential, hebey2000nonlinearanaly, kristaly2010variationalprinciples} for related definitions in Euclidean and Riemannian settings. Let $U$ be an open subset of $M$ and let $\qxp\geq1$. For $i=0,1,2$, define
\[
L^{\qxp}(U;\wghtv_i)
:=
\left\{
 u:U\rightarrow\mathbb{R}
 \mid
 u\text{ is measurable and }
 \int_U|u|^{\qxp}\,\dwv_i<\infty
\right\}.
\]

We denote
\begin{gather*}
{ 
\mathcal{C}^{k,\qxp}(\Omega;\wghtv_0,\ldots,\wghtv_k)
:=
\left\{
 u\in C^k(\Omega)
 \mid
 \sum_{i=0}^k\int_\Omega|\nabla^iu|^{\qxp}\,\dwv_i<\infty
\right\},
\qquad k=1,2,
}
\\[5pt]
{ 
C_0^\infty(\Omega)
:=
\left\{
 u\in C^\infty(\Omega)
 \mid
 \spt u\Subset\Omega
\right\},
}
\\[5pt]
C_0^\infty(\overline\Omega)
:=
\left\{
 u\in C^\infty(\Omega)
 \mid
 \text{there exists }\overline u\in C_0^\infty(M)
 \text{ such that }u=\overline u|_\Omega
\right\},
\\[5pt]
{ 
L^{\qxp}_{\loc}(\overline\Omega)
:=
\left\{
 u:\Omega\rightarrow\mathbb{R}
 \mid
 u\in L^{\qxp}(U\cap\Omega)
 \text{ for every open set }U\Subset M
\right\}.
}
\end{gather*}
Recall that $\wghtv_0$ and $\wghtv_1$ are bounded from above on compact subsets of $M$.

\begin{definition}$ $
\begin{enumerate}[label=(\roman*)]

{ 
\item The Sobolev space $W^{k,\qxp}(\Omega;\wghtv_0,\ldots,\wghtv_k)$, $k=1,2$, is the completion of $\mathcal{C}^{k,\qxp}(\Omega;\wghtv_0,\ldots,\wghtv_k)$ with respect to the norm
\[
\|u\|_{k,\qxp,\Omega,\wghtv_0,\ldots,\wghtv_k}
:=
\left(
\sum_{i=0}^k\int_\Omega|\nabla^iu|^{\qxp}\,\dwv_i
\right)^{\frac1{\qxp}}.
\]
}

{ 
\item The Sobolev space $W_0^{k,\qxp}(\overline\Omega;\wghtv_0,\ldots,\wghtv_k)$, $k=1,2$, is the completion of $C_0^\infty(\overline\Omega)$ with respect to the norm $\|\cdot\|_{k,\qxp,\Omega,\wghtv_0,\ldots,\wghtv_k}$.
}

{ 
\item The Sobolev space $W_0^{k,\qxp}(\Omega;\wghtv_0,\ldots,\wghtv_k)$, $k=1,2$, is the completion of $C_0^\infty(\Omega)$ with respect to the norm $\|\cdot\|_{k,\qxp,\Omega,\wghtv_0,\ldots,\wghtv_k}$.
}
\end{enumerate}
\end{definition}

\begin{definition}
We say that $u\in\womd$ is a weak solution of \eqref{61} if
\[
\int_\Omega g(\nabla u,\nabla v)\,\dwv_1
+
\int_\Omega uv\,\dwv_0
=
\lambda\int_\Omega\iwe uv\,\dwv_2
\]
for every $v\in\womd$.
\end{definition}


\section{Embeddings of Weighted Sobolev Spaces} \label{142}

In this section, we study continuous and compact embeddings and traces
for the weighted Sobolev space $W^{1,\qxp}(\Omega;\wghtv_0,\wghtv_1)$.

The proofs combine local Euclidean inequalities with the geometry of
the charts and the behavior of the weights. The quantities
\(\mathcal B_{\qxp,\qxp_0}^m\) and
\(\mathcal B_{\Gamma,\qxp,\qxp_1}^m\) control the behavior of functions
in the noncompact part of \(\Omega\). The section begins with an
extension result, followed by the embedding and trace theorems.

\begin{lemma} \label{6}
Let $(k,i)\in \{0,1\}\times \mathbb{N}$ and $\delta \geq 2$. There exists a smooth function $\zeta : \hat U_{k,i} \rightarrow [0,1]$ such that
$$
\zeta = 1 \text { in } U_{k,i}, \quad \spt \zeta \subset \psi_{k,i}\left(B\left(0,r_{k,i}+\delta^{-1}(\hat r_{k,i}-r_{k,i})\right)\right),
$$
and
$$
|\nabla \zeta|_g \leq C(n)\bnorm d(\psi_{k,i}^{-1})\bnorm\delta(\hat r_{k,i}-r_{k,i})^{-1} \quad \text { in } \hat U_{k,i}.
$$
\end{lemma}

\begin{proof}
There exists a smooth function $\zeta_0:B(0,\hat r_{k,i})\subset\mathbb R^n\rightarrow[0,1]$ such that
$$
\zeta_0=1 \text { in } B(0,r_{k,i}), \quad \spt\zeta_0\subset B\left(0,r_{k,i}+\delta^{-1}(\hat r_{k,i}-r_{k,i})\right),
$$
and
$$
|\nabla\zeta_0|_{\mec}\leq C_1(n)\delta(\hat r_{k,i}-r_{k,i})^{-1}.
$$
Define $\zeta:=\zeta_0\circ\psi_{k,i}^{-1}$. Then
$$
\zeta=1 \text { in } U_{k,i}, \quad \spt\zeta\subset\psi_{k,i}\left(B\left(0,r_{k,i}+\delta^{-1}(\hat r_{k,i}-r_{k,i})\right)\right),
$$
and
\begin{equation*}
\begin{aligned}
|\nabla\zeta|_g
\leq &\,\displaystyle \bnorm d(\psi_{k,i}^{-1})\bnorm\cdot|\nabla\zeta_0|_{\mec}\
\leq &\,\displaystyle C_1(n)\bnorm d(\psi_{k,i}^{-1})\bnorm\delta(\hat r_{k,i}-r_{k,i})^{-1}.
\end{aligned}
\end{equation*}
This proves the lemma.
\end{proof}

\begin{proposition} \label{22}
Assume that conditions \ref{5}--\ref{3} hold for $\qxp\geq1$, and that the weight functions $\wghtv_0$ and $\wghtv_1$ satisfy conditions \ref{138}--\ref{139}. Let $\delta\geq2$. Then there exists a bounded linear \textnormal{extension operator}
$$
\mathcal E:W^{1,\qxp}(\Omega;\wghtv_0,\wghtv_1)\rightarrow W^{1,\qxp}(M;\wghtv_0,\wghtv_1)
$$
such that
\begin{gather*}
\mathcal Eu=u \text { a.e. in } \Omega,\\[5pt]
\spt\mathcal Eu\subset\bigcup_{(k,i)\in\{0,1\}\times\mathbb N}\psi_{k,i}\left(B\left(0,r_{k,i}+\delta^{-1}(\hat r_{k,i}-r_{k,i})\right)\right),
\end{gather*}
and
\begin{equation}\label{9}
\begin{aligned}
\|\mathcal Eu&\|_{1,\qxp,M,\wghtv_0,\wghtv_1}\\[3pt]
\leq &\,\displaystyle C(n,\qxp)\rd_1^{\frac{q-1}{q}}\left[(1+\ck_0\rd_2)(1+\delta^{\qxp}\ck_{\qxp})\int_\Omega|u|^{\qxp}\,\dwv_0
+\left(1+\ck_0\rd_2^{1+\qxp}\right)\int_\Omega|\nabla u|^{\qxp}\,\dwv_1\right]^{\frac1\qxp}.
\end{aligned}
\end{equation}
\end{proposition}

\begin{proof}

Using Lemma \ref{6}, there exists a partition of unity $\{\zeta _{k,i} : M \rightarrow [0,1]\} _{(k,i) \in \{0,1\} \times \mathbb{N}} \subset C^{\infty}(M)$   subordinate to $\{ \hat{U}_{k,i} \}_{ (k,i) \in \{0,1\} \times \mathbb{N} }$ such that  
$$
\spt \zeta _{k,i} \subset \psi _{k,i} \left( B\left(0 ,  r _{k,i} + \delta ^{-1}(\hat r _{k,i} -r _{k,i} ) \right) \right),    \quad \sum_{k,i} \zeta _{k,i}=1 \text { on }\Omega,
$$ 
and 
$$ 
|\nabla \zeta _{k,i} |_g \leq  C_1(n) \delta \kp _{k , i} \quad \text {  in } \quad \hat U _{k,i}.
$$



Define $f:\mathbb R^n\rightarrow\mathbb R^n$ by $f(x)=(x_1,x_2,\ldots,|x_n|)$. Let $u\in W^{1,\qxp}(\Omega;\wghtv_0,\wghtv_1)$. For $(k,i)\in\{0,1\}\times\mathbb N$, define
\begin{gather*}
\overline u_{0,i}:=u\zeta_{0,i}\quad\text { in }\hat U_{0,i},\\[3pt]
\overline u_{1,j}(\pvr):=\left(u\circ\psi_{1,j}\circ f\circ\psi_{1,j}^{-1}\right)(\pvr)\zeta_{1,j}(\pvr)
\quad\text { if }\pvr\in\hat U_{1,j}.
\end{gather*}

By \ref{138} and \ref{3},
\begin{align}
\int_M|\overline u_{1,j}|^{\qxp}\,\dwv_0
=&\,\displaystyle\int_{\psi_{1,j}(B(0,\hat r_{1,j})\cap\{x_n>0\})}|\overline u_{1,j}|^{\qxp}\,\dwv_0
+\int_{\psi_{1,j}(B(0,\hat r_{1,j})\cap\{x_n<0\})}|\overline u_{1,j}|^{\qxp}\,\dwv_0\nonumber\\[3pt]
\leq&\,\displaystyle(1+\ck_0\rd_2)\int_{\hat U_{1,j}\cap\Omega}|u|^{\qxp}\,\dwv_0.\label{7}
\end{align}

Furthermore, from $|a+b|^{\qxp}\leq2^{\qxp-1}(|a|^{\qxp}+|b|^{\qxp})$, together with  \ref{3} and \ref{138}--\ref{139},
\begin{align}
\int_{M}|\nabla \overline u _{1,j}|^\qxp \, \dwv _1 
= & \displaystyle \, \int _{\psi _{1,j} (B(0,\hat r  _{1,j}) \cap \{x_n >0\})} |\nabla \overline u _{1,j}|^\qxp \, \dwv _1 + \int _{\psi _{1,j} (B(0,\hat r  _{1,j}) \cap \{x_n < 0\})} |\nabla \overline u _{1,j}|^\qxp \, \dwv _1 \nonumber\\[3pt]
 \leq & \displaystyle \,  \int _{\psi _{1,j} (B(0,\hat r  _{1,j}) \cap \{x_n >0\})} 2^{\qxp-1} \left( | \zeta _{1,j} \nabla u|^\qxp  + |u \nabla \zeta _{1,j} |^\qxp \right) \, \dwv _1  \nonumber \\[3pt]
& \displaystyle \, + \int _{\psi _{1,j} (B(0,\hat r  _{1,j}) \cap \{x_n > 0\})} 2^{\qxp-1} \ck _0 \rd _2 \left( |\rd _2 \zeta _{1,j} \nabla u |^\qxp  + |u \nabla \zeta _{1,j} |^\qxp \right) \, \dwv _1 \nonumber\\[5pt]
 \leq & \displaystyle \,  C_2 (n , \qxp ) \left[\delta ^{\qxp} \ck _{\qxp} (1+ \ck _0 \rd _{2})\int_{\hat U _{1,j} \cap \Omega }|u|^\qxp \, \dwv _0 +  \left(1+ \ck _0 \rd _2 ^{1+\qxp} \right)\int_{\hat U _{1,j} \cap \Omega }|\nabla u|^\qxp \, \dwv _1 \right]. \label{8}
\end{align}

{
For the interior charts, the same computation without reflection gives
\begin{align*}
\int_M|\overline u_{0,i}|^{\qxp}\,\dwv_0
&\leq\int_{\hat U_{0,i}}|u|^{\qxp}\,\dwv_0,\\[3pt]
\int_M|\nabla\overline u_{0,i}|^{\qxp}\,\dwv_1
&\leq C_3(n,\qxp)\left[\delta^{\qxp}\ck_{\qxp}\int_{\hat U_{0,i}}|u|^{\qxp}\,\dwv_0
+\int_{\hat U_{0,i}}|\nabla u|^{\qxp}\,\dwv_1\right].
\end{align*}
}

Define
$$
\mathcal Eu:=\sum_{(k,i)\in\{0,1\}\times\mathbb N}\overline u_{k,i}.
$$
Then $\mathcal E$ is linear and $\mathcal Eu=u$ a.e. in $\Omega$. { Since at most $\rd_1$ terms are nonzero at each point,}
$$
{ \left|\sum_{k,i}\overline u_{k,i}\right|^{\qxp}\leq\rd_1^{\qxp-1}\sum_{k,i}|\overline u_{k,i}|^{\qxp}.}
$$
Hence, by \ref{1}, \eqref{7}, and \eqref{8}, we obtain \eqref{9}.
\end{proof}

\subsection{Compact Embeddings}

The proof of our embedding theorems is based on the following lemma.
\begin{lemma} \label{21}
$ $
\begin{enumerate}[label=(\roman*)]
\item \label{18} Assume that
\begin{gather}
W^{1,\qxp}(\Omega_m;\wghtv_0,\wghtv_1)\rightarrow L^{\qxp_0}(\Omega_m;\wghtv_2) \text { is compact for every }m,\label{10}\\[3pt]
\lim_{m\rightarrow\infty}\sup_{\|u\|_{1,\qxp,\Omega,\wghtv_0,\wghtv_1}\leq1}\|u\|_{\qxp_0,\Omega^m,\wghtv_2}=0.\label{11}
\end{gather}
Then
$$
W^{1,\qxp}(\Omega;\wghtv_0,\wghtv_1)\rightarrow L^{\qxp_0}(\Omega;\wghtv_2)
$$
is compact.

On the other hand, if
$$
W^{1,\qxp}(\Omega;\wghtv_0,\wghtv_1)\rightarrow L^{\qxp_0}(\Omega;\wghtv_2)
$$
is compact, then \eqref{11} holds.

\item \label{19} If
\begin{gather}
W^{1,\qxp}(\Omega_m;\wghtv_0,\wghtv_1)\rightarrow L^{\qxp_0}(\Omega_m;\wghtv_2) \text { is continuous for every }m,\label{12}\\[5pt]
\lim_{m\rightarrow\infty}\sup_{\|u\|_{1,\qxp,\Omega,\wghtv_0,\wghtv_1}\leq1}\|u\|_{\qxp_0,\Omega^m,\wghtv_2}<\infty.\label{13}
\end{gather}
Then
$$
W^{1,\qxp}(\Omega;\wghtv_0,\wghtv_1)\rightarrow L^{\qxp_0}(\Omega;\wghtv_2)
$$
is continuous.

If
$$
W^{1,\qxp}(\Omega;\wghtv_0,\wghtv_1)\rightarrow L^{\qxp_0}(\Omega;\wghtv_2)
$$
is continuous, then \eqref{13} holds.
\end{enumerate}
\end{lemma}

\begin{proof}
$\ref{18}$ Notice that \eqref{11} is equivalent to the statement that for every $\epsilon>0$ there exists $m_0$ such that
\begin{equation}\label{17}
\|u\|_{\qxp_0,\Omega,\wghtv_2}^{\qxp_0}
\leq\epsilon\|u\|_{1,\qxp,\Omega,\wghtv_0,\wghtv_1}^{\qxp_0}
+\|u\|_{\qxp_0,\Omega_{m_0},\wghtv_2}^{\qxp_0}
\quad\forall u\in W^{1,\qxp}(\Omega;\wghtv_0,\wghtv_1).
\end{equation}

Let $(u_j)$ be a bounded sequence in $W^{1,\qxp}(\Omega;\wghtv_0,\wghtv_1)$, with
$$
\|u_j\|_{1,\qxp,\Omega,\wghtv_0,\wghtv_1}\leq C_1.
$$
{By \eqref{10} and a diagonal argument, there exists a subsequence, still denoted by $(u_j)$, which is a Cauchy sequence in $L^{\qxp_0}(\Omega_m;\wghtv_2)$ for every $m\in\mathbb N$.}
Let $\epsilon>0$. From \eqref{17}, for sufficiently large $i,j$,
\begin{align*}
\|u_j-u_i\|_{\qxp_0,\Omega,\wghtv_2}^{\qxp_0}
\leq&\,\displaystyle\epsilon\|u_j-u_i\|_{1,\qxp,\Omega,\wghtv_0,\wghtv_1}^{\qxp_0}
+\|u_j-u_i\|_{\qxp_0,\Omega_{m_0},\wghtv_2}^{\qxp_0}\\[3pt]
\leq&\,\displaystyle\left(2^{\qxp_0}C_1^{\qxp_0}+1\right)\epsilon.
\end{align*}
Consequently, $(u_j)$ is a Cauchy sequence in $L^{\qxp_0}(\Omega;\wghtv_2)$. Hence the embedding is compact.

Now, let the embedding be compact and assume that \eqref{17} does not hold. Then there exist $\epsilon>0$ and a sequence $(u_j)$ in $W^{1,\qxp}(\Omega;\wghtv_0,\wghtv_1)$ such that
$$
\|u_j\|_{\qxp_0,\Omega,\wghtv_2}^{\qxp_0}
>\epsilon\|u_j\|_{1,\qxp,\Omega,\wghtv_0,\wghtv_1}^{\qxp_0}
+\|u_j\|_{\qxp_0,\Omega_j,\wghtv_2}^{\qxp_0}.
$$
Write
$$
\tilde u_j:=\frac{u_j}{\|u_j\|_{1,\qxp,\Omega,\wghtv_0,\wghtv_1}}.
$$
Then
\begin{equation}\label{20}
\|\tilde u_j\|_{\qxp_0,\Omega,\wghtv_2}^{\qxp_0}
>\epsilon+\|\tilde u_j\|_{\qxp_0,\Omega_j,\wghtv_2}^{\qxp_0}.
\end{equation}
Since $(\tilde u_j)$ is bounded in $W^{1,\qxp}(\Omega;\wghtv_0,\wghtv_1)$, there is a subsequence converging to $\tilde u$ in $L^{\qxp_0}(\Omega;\wghtv_2)$. Moreover,
$$
\left|\|\tilde u_j\|_{\qxp_0,\Omega_j,\wghtv_2}-\|\tilde u\|_{\qxp_0,\Omega,\wghtv_2}\right|
\leq\|\tilde u_j-\tilde u\|_{\qxp_0,\Omega,\wghtv_2}
+\left|\|\tilde u\|_{\qxp_0,\Omega_j,\wghtv_2}-\|\tilde u\|_{\qxp_0,\Omega,\wghtv_2}\right|.
$$
Taking the limit in \eqref{20}, we obtain
$$
\|\tilde u\|_{\qxp_0,\Omega,\wghtv_2}^{\qxp_0}
\geq\epsilon+\|\tilde u\|_{\qxp_0,\Omega,\wghtv_2}^{\qxp_0},
$$
which is a contradiction.

\medskip

$\ref{19}$ The second part of the lemma can be proved similarly to $\ref{18}$.
\end{proof}

\subsubsection{Proof of Proposition \ref{33}}

We have
$$
\|v\|_{\qxp_0,B(0,1)}\leq C_1\|v\|_{1,\qxp,B(0,1)\cap\{x_n>0\}},
\quad\forall v\in W^{1,\qxp}(B(0,1)\cap\{x_n>0\}),
$$
where $B(0,1)\subset \mathbb{R}^n$ and $C_1:=C_1(\qxp , \qxp _0 , n , B(0,1)) >0 $.

Let $u\in W^{1,\qxp}(\Omega;\wghtv_0,\wghtv_1)$. Then,
\begin{equation*}
\begin{aligned}
\left(\int_{\hat U_{k,i}\cap\Omega}|u|^{\qxp_0}\,\dvg\right)^{\frac1{\qxp_0}}
= &\,\displaystyle \left(\int_{ B(0 , \hat r_{k,i} )\cap \{x_n >0\} } | u \circ \psi _{k,i} |^{ \qxp _0 }  \sqrt{\det [ g_{\gij} ]} \, \dx \right)^{\frac{1}{ \qxp _0 }}\\[3pt]
  \leq & \displaystyle \,  C_1  \bnorm G_{k,i}\bnorm ^{\frac{1}{ \qxp _0 }}  \hat{r}_{k,i}^{\frac{n}{ \qxp _0  }} \left(\int_{ B(0 , 1 ) \cap \{x_n >0\} } | u \circ \psi _{k,i}(\hat r _{k,i} y) |^{ \qxp _0 } \,  \dy \right)^{\frac{1}{   \qxp _0}  }\\[5pt]
\leq&\,\displaystyle C_2  \bnorm G_{k,i}\bnorm^{\frac1{\qxp_0}}
\bnorm G_{k,i}^{-1}\bnorm^{\frac1\qxp}\hat r_{k,i}^{\frac n{\qxp_0}}\\[3pt]
&\,\displaystyle\cdot\left(\hat r_{k,i}^{-n}\int_{\hat U_{k,i}\cap\Omega}|u|^{\qxp}\,\dvg
+\hat r_{k,i}^{-n+\qxp}\bnorm d\psi_{k,i}\bnorm^{\qxp}
\int_{\hat U_{k,i}\cap\Omega}|\nabla u|^{\qxp}\,\dvg\right)^{\frac1\qxp},
\end{aligned}
\end{equation*}
where $C_2:= C_1 n^{1/2}$.

From  \ref{139} and \ref{46},
\begin{equation}\label{155}
\begin{aligned}
\int_{\hat U_{k,i}\cap\Omega}|u|^{\qxp_0}\,\dwv_2 \leq & \, \displaystyle  \left[b_2  ^{\frac{1}{ \qxp _0 }} (\psi _{k,i} (0)) C_2 \bnorm G_{k,i}\bnorm ^{\frac{1}{  \qxp _0 }}  \bnorm G_{k,i} ^{-1}\bnorm ^{\frac{1}{ \qxp}}   \hat r_{k,i}^{\frac{n}{ \qxp _0   }-\frac{n}{\qxp }+1} \right] ^{ \qxp _0 }   \\[3pt]
& \, \displaystyle   \cdot \left(\hat r_{k,i}^{-\qxp  }\int_{\hat U _{k,i}}| u |^{\qxp} \, \dvg + \bnorm d \psi _{k,i}\bnorm ^{\qxp}  \int_{\hat U _{k,i}}|\nabla  u  |^{\qxp}  \, \dvg \right)^{\frac{ \qxp _0 }{ \qxp}  } \\[3pt]
\leq&\,\displaystyle C_2^{\qxp_0}
\left[\frac{b_2^{\frac1{\qxp_0}}(\psi_{k,i}(0))}{b_1^{\frac1\qxp}(\psi_{k,i}(0))}
\bnorm G_{k,i}\bnorm^{\frac1{\qxp_0}}
\bnorm G_{k,i}^{-1}\bnorm^{\frac1\qxp}
\hat r_{k,i}^{\frac n{\qxp_0}-\frac n\qxp+1}\right]^{\qxp_0}\\[3pt]
&\,\displaystyle\cdot\left(\hat r_{k,i}^{-\qxp}\bnorm d\psi_{k,i}\bnorm^{-\qxp}
\int_{\hat U_{k,i}\cap\Omega}|u|^{\qxp}\,\dwv_1
+\int_{\hat U_{k,i}\cap\Omega}|\nabla u|^{\qxp}\,\dwv_1\right)^{\frac{\qxp_0}\qxp}\\[3pt]
\leq&\,\displaystyle C_2^{\qxp_0}(\mathcal B_{\qxp,\qxp_0}^m)^{\qxp_0}
\left(\ck_{\qxp}\int_{\hat U_{k,i}\cap\Omega}|u|^{\qxp}\,\dwv_0
+\int_{\hat U_{k,i}\cap\Omega}|\nabla u|^{\qxp}\,\dwv_1\right)^{\frac{\qxp_0}\qxp},
\end{aligned}
\end{equation}
where condition \ref{139} was used in the last inequality.

Since
$$
\chi_{\Omega}
\leq\sum_{(k,i)\in\{0,1\}\times\mathbb N} \chi_{\hat U_{k,i}\cap \Omega} \leq \rd _1 \chi _{\Omega},
$$
from \ref{1}, \eqref{155}, and
$$
\sum_i|a_i|^\kappa\leq\left(\sum_i|a_i|\right)^\kappa,
\quad \kappa\geq1,
$$
we obtain
\begin{align*}
\|u\|_{\qxp_0,\Omega^{m},\wghtv_2}^{\qxp_0}
\leq&\,\displaystyle C_2 ^{\qxp_0}(\mathcal B_{\qxp,\qxp_0}^m)^{\qxp_0}
\left[\sum_{k,i}\left(\ck_{\qxp}\int_{\hat U_{k,i}\cap\Omega}|u|^{\qxp}\,\dwv_0
+\int_{\hat U_{k,i}\cap\Omega}|\nabla u|^{\qxp}\,\dwv_1\right)\right]^{\frac{\qxp_0}\qxp}\\[3pt]
\leq&\,\displaystyle C_2 ^{\qxp_0}(\mathcal B_{\qxp,\qxp_0}^m)^{\qxp_0}
[\rd_1(\ck_{\qxp}+1)]^{\frac{\qxp_0}\qxp}
\|u\|_{1,\qxp,\Omega,\wghtv_0,\wghtv_1}^{\qxp_0}.
\end{align*}
Thus,
\begin{equation}\label{24}
\|u\|_{\qxp_0,\Omega^{m},\wghtv_2}
\leq C_2 \mathcal B_{\qxp,\qxp_0}^m[\rd_1(\ck_{\qxp}+1)]^{\frac1\qxp}
\|u\|_{1,\qxp,\Omega,\wghtv_0,\wghtv_1}.
\end{equation}

Recall that, on every compact subset of $M$, the weights are bounded above and below by positive constants. Therefore, the local embedding in \eqref{12} is continuous when $\qxp_0\leq n\qxp/(n-\qxp)$, whereas the local embedding in \eqref{10} is compact when $\qxp_0<n\qxp/(n-\qxp)$. The proposition now follows from Lemma \ref{21} and \eqref{24}.

\begin{flushright}
$\square$
\end{flushright}

\subsection{Compact Trace Results}
Proceeding similarly to Lemma \ref{21}, we have the following lemma.
\begin{lemma} \label{25}
$ $
\begin{enumerate}[label=(\roman*)]
\item \label{26} Assume that
\begin{gather}
W^{1,\qxp}(\Omega_m;\wghtv_0,\wghtv_1)\rightarrow L^{\qxp_1}(\Gamma_m;\wghtw) \text { is compact for every }m,\label{27}\\[5pt]
\lim_{m\rightarrow\infty}\sup_{\|u\|_{1,\qxp,\Omega,\wghtv_0,\wghtv_1}\leq1}\|u\|_{\qxp_1,\Gamma^m,\wghtw}=0.\label{28}
\end{gather}
Then
$$
W^{1,\qxp}(\Omega;\wghtv_0,\wghtv_1)\rightarrow L^{\qxp_1}(\Gamma;\wghtw)
$$
is compact.

On the other hand, if the trace operator is compact, then \eqref{28} holds.

\item \label{29} If
\begin{gather}
W^{1,\qxp}(\Omega_m;\wghtv_0,\wghtv_1)\rightarrow L^{\qxp_1}(\Gamma_m;\wghtw) \text { is continuous for every }m,\label{30}\\[5pt]
\lim_{m\rightarrow\infty}\sup_{\|u\|_{1,\qxp,\Omega,\wghtv_0,\wghtv_1}\leq1}\|u\|_{\qxp_1,\Gamma^m,\wghtw}<\infty.\label{31}
\end{gather}
Then
$$
W^{1,\qxp}(\Omega;\wghtv_0,\wghtv_1)\rightarrow L^{\qxp_1}(\Gamma;\wghtw)
$$
is continuous.

If the trace operator is continuous, then \eqref{31} holds.
\end{enumerate}
\end{lemma}



\subsubsection{Proof of Proposition \ref{56}}

\noindent {\bf Step 1.} We begin by defining the trace operator. For every
\(j\in\mathbb N\), set
\[
A_j:=\hat U_{1,j}\cap\Gamma.
\]
Let
\[
\operatorname{Tr}_{1,j}:
W^{1,\qxp}
\left(
B(0,\hat r_{1,j})\cap\{x_n>0\}
\right)
\longrightarrow
L^{\qxp_1}
\left(
B(0,\hat r_{1,j})\cap\{x_n=0\}
\right)
\]
be the Euclidean trace operator. For
\(u\in W^{1,\qxp}(\Omega;\wghtv_0,\wghtv_1)\), define the local trace
\(T_j u\) on \(A_j\) by
\[
(T_j u)\circ\psi_{1,j}
:=
\operatorname{Tr}_{1,j}(u\circ\psi_{1,j}).
\]

This definition is well posed because the weights are bounded above
and below by positive constants on compact subsets of \(M\). Hence
\(u\circ\psi_{1,j}\) belongs to the corresponding Euclidean Sobolev
space.

If \(A_j\cap A_l\neq\varnothing\), the transition map
\[
\psi_{1,l}^{-1}\circ\psi_{1,j}
\]
is a smooth change of boundary coordinates. Since the Euclidean trace
commutes with smooth changes of coordinates preserving the flat
boundary, we have
\[
T_j u=T_l u
\quad\text{a.e. on }A_j\cap A_l.
\]

To define the global trace, set
\[
E_1:=A_1,
\qquad
E_j:=A_j\setminus\bigcup_{l=1}^{j-1}A_l,
\quad j\geq2.
\]
The sets \(E_j\) are measurable, pairwise disjoint, and
\[
\Gamma=\bigcup_{j\in\mathbb N}E_j.
\]

We define
\[
Tu:=\sum_{j=1}^{\infty}\chi_{E_j}T_j u.
\]

Thus, \(Tu\) is measurable. Moreover, the compatibility of the local
traces gives
\[
Tu=T_j u
\quad\text{a.e. on }A_j
\]
for every \(j\in\mathbb N\). Since the family
\(\{A_j\}_{j\in\mathbb N}\) is countable and covers \(\Gamma\), this
property also shows that \(Tu\) is unique up to equality almost
everywhere on \(\Gamma\).

In addition, if
\[
u\in C(\overline\Omega)\cap
W^{1,\qxp}(\Omega;\wghtv_0,\wghtv_1),
\]
then
\[
Tu=u|_\Gamma
\quad\text{a.e. on }\Gamma.
\]



\medskip

\noindent {\bf Step 2.} We now prove the continuity and compactness assertions for the trace operator.

For simplicity, we identify \(Tu\) with \(u\) on \(\Gamma\); that is, we write
\[
u:=Tu
\qquad\text{on }\Gamma.
\]

We have
$$
\|v\|_{\qxp_1,B(0,1)\cap\{x_n=0\}}
\leq c_1\|v\|_{1,\qxp,B(0,1)\cap\{x_n>0\}},
\quad\forall v\in W^{1,\qxp}(B(0,1)\cap\{x_n>0\}),
$$
where $B(0,1)\subset \mathbb{R}^n$ and $c_1 := c_1(\qxp , \qxp _1 , n , B(0,1))>0$.

For $u\in W^{1,\qxp}(\Omega;\wghtv_0,\wghtv_1)$, we obtain
\begin{align*}
\left(\int_{\hat U_{1,j}\cap\Gamma}|u|^{\qxp_1}\,\dsg\right)^{\frac1{\qxp_1}} = & \,  \displaystyle   \left(\int_{ B(0 , \hat r_{1,j} )\cap \{x_n =0\}  } | u \circ \psi _{1,j} |^{\qxp _1}  \sqrt{\det [ g_{\Gamma, \gij}]} \, \ds \right)^{\frac{1}{\qxp _1}}\\[3pt]
  \leq & \, \displaystyle   c _1 \bnorm  G_{\Gamma , 1,j} \bnorm ^{\frac{1}{\qxp _1}}  \hat{r}_{1,j}^{\frac{n-1}{\qxp _1 }} \left(\int_{ B(0 , 1 )\cap \{x_n >0\}} | u \circ \psi _{1,j}(\hat r _{1,j} y) |^{\qxp }   \dy \right)^{\frac{1}{\qxp}} \\[5pt]
\leq&\,\displaystyle c_2 \bnorm G_{\Gamma,1,j}\bnorm^{\frac1{\qxp_1}}
\bnorm G_{1,j}^{-1}\bnorm^{\frac1\qxp}\hat r_{1,j}^{\frac{n-1}{\qxp_1}}\\[3pt]
&\,\displaystyle\cdot\left(\hat r_{1,j}^{-n}\int_{\hat U_{1,j}\cap\Omega}|u|^{\qxp}\,\dvg
+\hat r_{1,j}^{-n+\qxp}\bnorm d\psi_{1,j}\bnorm^{\qxp}
\int_{\hat U_{1,j}\cap\Omega}|\nabla u|^{\qxp}\,\dvg\right)^{\frac1\qxp},
\end{align*}
where $c_2 := c_1 n^{1/2}$.

Proceeding as in the proof of Proposition \ref{33}, it follows from \ref{16} and \ref{139} that
\begin{align*}
\int_{\hat U_{1,j}\cap\Gamma}|u|^{\qxp_1}\,\dww
\leq&\,\displaystyle c_2 ^{\qxp_1}(\mathcal B_{\Gamma,\qxp,\qxp_1}^m)^{\qxp_1}
\left(\ck_{\qxp}\int_{\hat U_{1,j}\cap\Omega}|u|^{\qxp}\,\dwv_0
+\int_{\hat U_{1,j}\cap\Omega}|\nabla u|^{\qxp}\,\dwv_1\right)^{\frac{\qxp_1}\qxp}.
\end{align*}

From \ref{1},
\begin{align*}
\|u\|_{\qxp_1,\Gamma^{m},\wghtw}^{\qxp_1}
\leq&\,\displaystyle c_2 ^{\qxp_1}(\mathcal B_{\Gamma,\qxp,\qxp_1}^m)^{\qxp_1}
\left[\sum_j\left(\ck_{\qxp}\int_{\hat U_{1,j}\cap\Omega}|u|^{\qxp}\,\dwv_0
+\int_{\hat U_{1,j}\cap\Omega}|\nabla u|^{\qxp}\,\dwv_1\right)\right]^{\frac{\qxp_1}\qxp}\\[3pt]
\leq&\,\displaystyle c_2 ^{\qxp_1}(\mathcal B_{\Gamma,\qxp,\qxp_1}^m)^{\qxp_1}
[\rd_1(\ck_{\qxp}+1)]^{\frac{\qxp_1}\qxp}
\|u\|_{1,\qxp,\Omega,\wghtv_0,\wghtv_1}^{\qxp_1}.
\end{align*}
Thus,
\begin{equation}\label{35}
\|u\|_{\qxp_1,\Gamma^{m},\wghtw}
\leq c_2 \mathcal B_{\Gamma,\qxp,\qxp_1}^m[\rd_1(\ck_{\qxp}+1)]^{\frac1\qxp}
\|u\|_{1,\qxp,\Omega,\wghtv_0,\wghtv_1}.
\end{equation}

Recall that the weights are bounded above and below by positive constants on every compact subset of \(M\). Therefore, the local trace operator in \eqref{30} is continuous when $\qxp_1\leq (n-1)\qxp/(n-\qxp)$,
  whereas the local trace operator in \eqref{27} is compact when $\qxp_1<(n-1)\qxp/(n-\qxp)$. The proposition then follows from Lemma \ref{25} and \eqref{35}.

\begin{flushright}
$\square$
\end{flushright}

\section{Principal Eigenvalues with Indefinite Weights}\label{143}


Before proving Theorem \ref{59}, we prove the preliminary Propositions \ref{52} and \ref{53}.

\subsection{Elliptic Estimates} 

Similarly to Lemma \ref{6}, we have the following result.
\begin{lemma} \label{153}
Suppose $0<r<R\leq \hat r _{1,j}$ and fix $j\in \mathbb{N}$. There exists a smooth function $\zeta  : \hat U  _{1,j} \rightarrow [0,1]$ such that  
$$
\zeta  = 1 \text { in } \psi _{1,j} (B(0,r)), \quad   \spt \zeta  \subset \psi _{1,j} (B(0,R)),
$$ 
and 
$$
|\nabla \zeta  | _g \leq \cen _1(n)\bnorm d (\psi _{1,j}^{-1}) \bnorm  (R-r)^{-1} \quad \text { in } \quad \hat  U _{1,j}.
$$ 
\end{lemma}

For $h \in \mathbb{R}$ and $t\in (0,\hat r_{1,j})$, we define
\begin{gather*}
\ops [u, v] :=\int  _{\Omega}    g(  \nabla u , \nabla v ) \, \dwv _{1}  + \int  _{\Omega}  u v  \,   \dwv _{0} ,\\[5pt]
\Psi (h ,t) := \sqrt{\int _{ \mathscr{U}(h ,t) } u_{ h }  ^2  \, \dwv _2 + \int _{ \mathscr{U}_{\Gamma }  (h ,t)  } u_{ h } ^2   \, \dww },
\end{gather*}
where $u,v\in \womd  $,
\begin{gather*}
u_{h}  := (u-h )^+,\\[5pt]
\mathscr{U}(h ,t) := \left\{\pvr \in \psi _{1,j} (  B(0,t)\cap \{x_n >0\} ) \mid u(\pvr) > h  \right\} , \\[5pt]
 \mathscr{U}_{ \Gamma } (h ,t) := \left\{ \pvr \in \psi _{1,j} (  B(0,t)\cap \{x_n =0\} ) \mid u(\pvr) > h  \right\}.  
\end{gather*}

The proof of the next lemma will be given in the appendix.
\begin{lemma} \label{alcontb0}
Suppose $\kp _{1 , j} ^{2 } (\pvr) \wghtv _1 (\pvr )  \leq \ck _{2} \wghtv _2 (\pvr )$ for   a.e. $\pvr  \in \hat U _{1,j}$. Assume that   $\fc _2, f \in L^{q_2} (\hat U _{1,j} \cap \Omega  ; \wghtv _2)$,  and  $\fc _3 , f_1 \in L^{q_3} (\hat U _{1,j} \cap \Gamma ; \wghtw)$  for some  $q _2, q_3 > \max \{ \xv / (\xv -2) , \xw / (\xw -2) \}$. Suppose  
\begin{equation}\label{a168}
\begin{aligned}
 \ops [u,v_{h} ] + \int  _{\hat U _{1,j} \cap \Omega } \fc _2 u v_{h}   \, \dwv _{2}   + \int  _{ \hat U _{1,j} \cap \Gamma} \fc _3 u v_{ h } \dww  
 \leq \int  _{\hat U _{1,j} \cap \Omega }    f v_{ h } \, \dwv _{2}  + \int_{\hat U _{1,j} \cap \Gamma }  f_1 v _{ h } \, \dww ,  
\end{aligned}
\end{equation}
for all $h\geq 0 $, where $v _{ h }  = u_{ h }  \zeta ^2 $ and $\zeta$ is given in Lemma \ref{153}.

There exist  $\epsilon  = \epsilon ( q_2 , q_3 , \xv , \xw ) >0$ and  $\cen _2 = \cen _2 (  n , \mathcal C_{\xv} , \mathcal C_{\xw} , \xv , \xw , \qxp _2 , \qxp _3) >0$  such that if  
\begin{align*}
h_{2} > & \displaystyle \, h_{1}  \\[3pt]
  \geq & \displaystyle \,  \cen _2 \max \left\{ \widehat{\fc}_2 ^{\frac{\qxp _2 \xv }{\qxp _2 (\xv -2) - \xv}}   , \widehat{\fc}_3 ^{\frac{\qxp _3 \xw }{\qxp _3 (\xw -2) - \xw}} , 1\right\}  \max \left\{ \Vert u^+ \Vert _{2 , \hat U _{1,j} \cap \Omega  , \wghtv _2}  ,   \Vert  u^+ \Vert _{2 , \hat U _{1,j} \cap \Gamma , \wghtw }  \right\}, 
\end{align*}
then 
\begin{equation} \label{48}
\begin{aligned}
&\Psi  (h_{2},r) \\[3pt]
&\quad \ \leq \cen _3 \left[   \frac{  \hat r_{1,j}  - r_{1,j} }{(R-r)(h_{2}-h_{1} )^{\epsilon }}  + \frac{ \Vert f \Vert _{ q_2 ,  \hat U _{1,j} \cap \Omega  , \wghtv _2}  + \Vert f_1 \Vert _{q_3 , \hat U _{1,j} \cap \Gamma , \wghtw} + h_{2} (\widehat{\fc}_2 + \widehat{\fc}_3)^{1/2}}{(h_{2}-h_{1} )^{1+\epsilon }}\right] \Psi ^{1 + \epsilon  } (h_{1} ,R),
\end{aligned}
\end{equation}
where  $\cen _3= \cen _3  (  n , \mathcal C_{\xv} , \mathcal C_{\xw} ,  \xv , \xw , \qxp _2 , \qxp _3 , \ck _2) >0$, and
$$
\widehat{\fc}_2 :=\Vert \fc _2 \Vert _{q_2 , \hat U _{1,j} \cap \Omega  , \wghtv _2 } ,\qquad \widehat{\fc}_3 :=\Vert \fc _3 \Vert _{q_3 , \hat U _{1,j} \cap \Gamma  , \wghtw } .
$$

\end{lemma}


\begin{proposition} \label{52}
Suppose $\kp _{1 , j} ^{2 } (\pvr) \wghtv _1 (\pvr )  \leq \ck _{2} \wghtv _2 (\pvr )$ for   a.e. $\pvr  \in \hat U _{1,j}$.  Assume that   $\fc _2, f \in L^{q_2} (\hat U _{1,j} \cap \Omega  ; \wghtv _2)$,  and  $\fc _3 , f_1 \in L^{q_3} (\hat U _{1,j} \cap \Gamma ; \wghtw)$  for some  $q _2, q_3 > \max \{ \xv / (\xv -2) , \xw / (\xw -2) \}$. Suppose  
\begin{equation*}
\begin{aligned}
 \ops [u,v_{h} ] + \int  _{\Omega } \fc _2 u v_{h}   \,  \dwv _{2}   + \int  _{ \Gamma }\fc _3 u v_{ h } \, \dww  
 \leq \int  _{\Omega }    f v_{ h } \,  \dwv _{2}  + \int_{\Gamma } f_1 v _{ h } \, \dww ,  
\end{aligned}
\end{equation*}
for all $h\geq 0 $,   $v _{ h }  = u_{ h }  \zeta ^2 $ and $\zeta \in C^{\infty} _0 ( \hat U_{1,j} )$. Then
\begin{align*}
 \stackbin[ U _{1,j} \cap \Omega ]{}{\operatorname{ess}\sup} \,   u^{+} + & \stackbin[ U_{1,j} \cap \Gamma ]{}{\operatorname{ess}\sup} \, u^{+} \\[3pt]
 \leq & \displaystyle \,  C\left(1+ \max \left\{   \widehat{\fc}_2^{\frac{\qxp _2 \xv }{\qxp _2 (\xv -2) - \xv}}   , \widehat{\fc}_3 ^{\frac{\qxp _3 \xw }{\qxp _3 (\xw -2) - \xw}} \right\}  \right)\\[3pt]
& \displaystyle \, \cdot \left(\Vert u^+ \Vert _{2 , \hat U _{1,j} \cap \Omega   , \wghtv _2}  +   \Vert  u^+ \Vert _{2 , \hat U _{1,j} \cap \Gamma    , \wghtw }  \right) \\[3pt]
& \displaystyle \, + C\left( \Vert f \Vert _{ q_2 ,  \hat U _{1,j} \cap \Omega    , \wghtv _2}  + \Vert f_1 \Vert _{q_3 , \hat U _{1,j} \cap \Gamma  , \wghtw} \right)(1+\widehat{\fc}_2 + \widehat{\fc}_3)^{-1/2},
\end{align*}
where  $C=C (  n , \mathcal C_{\xv} , \mathcal C_{\xw} ,  \xv , \xw , \qxp _2 , \qxp _3 , \ck _2) >0$.
\end{proposition}

\begin{proof}  

We carry out the iteration.  Define for $i\in \mathbb{N} \cup \{0\}$,
$$
 h_{i}:=h_0 + h \left(1-\frac{1}{2^{i}}\right) \leq  h_0 +h \quad \text { and } \quad  R_{i}:= r _{1,j} +\frac{1}{2^{i}} (\hat r _{1,j} -  r _{1,j}).
$$
where  
$$
h_0 = \cen _2 \max \left\{ \widehat{\fc}_2 ^{\frac{\qxp _2 \xv }{\qxp _2 (\xv -2) - \xv}}   , \widehat{\fc}_3 ^{\frac{\qxp _3 \xw }{\qxp _3 (\xw -2) - \xw}} , 1\right\} \max \left\{ \Vert u^+ \Vert _{2 , \hat U _{1,j} \cap \Omega  , \wghtv _2}  ,   \Vert  u^+ \Vert _{2 , \hat U _{1,j} \cap \Gamma , \wghtw }   \right\},
$$
$\cen _2$ is given by Lemma \ref{alcontb0}, and $h>0$ will be determined later.

We have
$$
h_{i}-h_{i-1}=\frac{h}{2^{i}}\quad \text { and } \quad R_{i-1}-R_{i}=\frac{\hat r _{1,j} -  r _{1,j}}{2^{i}} .
$$

Hence, from \eqref{48}, 
\begin{align}
&\Psi  (h_{i}, R _i) \nonumber \\[3pt]
&\quad \ \leq \cen _3 \left[    2^i  + 2^i\frac{\Vert f \Vert _{ q_2 ,  \hat U _{1,j} \cap \Omega  , \wghtv _2}  + \Vert f_1 \Vert _{q_3 , \hat U _{1,j} \cap \Gamma , \wghtw} + (h _0 + h )(1+\widehat{\fc}_2 + \widehat{\fc}_3)^{1/2} }{h}\right]\frac{2^{\epsilon i}}{h^\epsilon} \Psi ^{1 + \epsilon  } (h_{i-1} , R_{i-1}) \nonumber \\[5pt]
&\quad \ \leq 2\cen _3 \left[ \Vert f \Vert _{ q_2 ,  \hat U _{1,j} \cap \Omega  , \wghtv _2}  + \Vert f_1 \Vert _{q_3 , \hat U _{1,j} \cap \Gamma , \wghtw} + (h _0 + h  )(1+\widehat{\fc}_2 + \widehat{\fc}_3)^{1/2}\right] \frac{2^{(1+\epsilon )i}}{h^{1+\epsilon}} \Psi ^{1 + \epsilon  } (h_{i-1} , R_{i-1}), \label{49}
\end{align}

Next we prove inductively for any $i\in \mathbb{N} \cup \{0\}$,
\begin{equation} \label{51}
\Psi  (h_{i} , R_{i}) \leq \frac{\Psi  (h_{0} , R_{0}) }{\gamma^{i}} \quad \text { for some } \quad \gamma>1,
\end{equation}
if $h$ is sufficiently large. It is true for $i=0$. Suppose it is true for $i-1$. We have
\begin{equation} \label{50}
\begin{aligned}
\Psi ^{1 + \epsilon  } (h_{i-1} , R_{i-1}) \leq & \, \displaystyle \left(  \frac{\Psi (h_0, R_0 )}{\gamma^{i-1}}\right) ^{1+\epsilon }\\[3pt]
=  & \, \displaystyle \frac{\Psi ^{\epsilon } (h_0, R_0)}{\gamma^{i \epsilon -(1+\epsilon )}} \cdot  \frac{\Psi (h_0, R_0)}{\gamma^{i}} .
\end{aligned}
\end{equation}

Then, by \eqref{49} and \eqref{50}, we obtain
\begin{align*}
\Psi (h_{i},  R_i) 
 \leq & \, \displaystyle  2\cen _3 \gamma^{1+\epsilon } \frac{  \Vert f \Vert _{ q_2 ,  \hat U _{1,j} \cap \Omega  , \wghtv _2}  + \Vert f_1 \Vert _{q_3 , \hat U _{1,j} \cap \Gamma , \wghtw} +( h _0 + h  )(1+\widehat{\fc}_2 + \widehat{\fc}_3)^{1/2} }{h^{1+\epsilon }} \\[3pt]
 & \, \displaystyle \cdot \Psi ^\epsilon (h_0, R_0)  \frac{2^{i (1+\epsilon )}}{\gamma^{i \epsilon }} \cdot \frac{\Psi ( h_0, R_0 )}{\gamma^{i}} .
\end{align*}

Choose $\gamma$ first such that $\gamma^{\epsilon }=2^{1+\epsilon }$. Note $\gamma>1$. Next, we need
$$
2\cen _3 \gamma^{1+\epsilon } \left(\frac{\Psi (h_0, R_0)}{h}\right)^{\epsilon }   \frac{  \Vert f \Vert _{ q_2 ,  \hat U _{1,j} \cap \Omega  , \wghtv _2}  + \Vert f_1 \Vert _{q_3 , \hat U _{1,j} \cap \Gamma , \wghtw} + (h _0 + h )(1+\widehat{\fc}_2 + \widehat{\fc}_3)^{1/2}  }{h} \leq 1 .
$$

Therefore, we choose
$$
h=C \left[   \left(\Vert f \Vert _{ q_2 ,  \hat U _{1,j} \cap \Omega  , \wghtv _2}  + \Vert f_1 \Vert _{q_3 , \hat U _{1,j} \cap \Gamma , \wghtw}\right)(1+\widehat{\fc}_2 + \widehat{\fc}_3)^{-1/2} + h _0   + \Psi ( h_0 , R_0) \right],
$$
for $C= C (  n , \mathcal C_{\xv} , \mathcal C_{\xw} ,  \xv , \xw , \qxp _2 , \qxp _3 , \ck _2) >0$ sufficiently large. This proves \eqref{51}.

Taking \( i \rightarrow \infty \) in \eqref{51}, we conclude
$$
\Psi (h_0 + h , r_{1,j} )=0.
$$

Hence, 
\begin{align*}
 \stackbin[ U_{1,j }\cap \Omega ]{}{\operatorname{ess}\sup} \,  u^{+} + & \stackbin[ U _{1,j} \cap \Gamma   ]{}{\operatorname{ess}\sup} \, u^{+} \\[3pt]
 \leq & \displaystyle \,  2  (C+1)\left[ \left( \Vert f \Vert _{ q_2 ,  \hat U _{1,j} \cap \Omega  , \wghtv _2}  + \Vert f_1 \Vert _{q_3 , \hat U _{1,j} \cap \Gamma , \wghtw} \right)(1+\widehat{\fc}_2 + \widehat{\fc}_3)^{-1/2}+ h _0   + \Psi ( h_0 , R_0) \right]\\[5pt]
 \leq & \displaystyle \,  2 (C+1)\Bigg[\left( \Vert f \Vert _{ q_2 ,  \hat U _{1,j} \cap \Omega  , \wghtv _2}  + \Vert f_1 \Vert _{q_3 , \hat U _{1,j} \cap \Gamma , \wghtw}\right)(1+\widehat{\fc}_2 + \widehat{\fc}_3)^{-1/2} \\[3pt]
& \displaystyle \, + \cen _2 \max \left\{ \widehat{\fc}_2 ^{\frac{\qxp _2 \xv }{\qxp _2 (\xv -2) - \xv}}   , \widehat{\fc}_3^{\frac{\qxp _3 \xw }{\qxp _3 (\xw -2) - \xw}} , 1 \right\}  \\[3pt]
& \displaystyle \, \cdot \max \left\{ \Vert u^+ \Vert _{2 , \hat U _{1,j} \cap \Omega  , \wghtv _2 }  ,   \Vert  u^+ \Vert _{2 , \hat U _{1,j} \cap \Gamma , \wghtw }   \right\}   \\[3pt]
& \displaystyle \, + \Vert u^+ \Vert _{2 , \hat U _{1,j} \cap \Omega  , \wghtv _2}  +   \Vert  u^+ \Vert _{2 , \hat U _{1,j} \cap \Gamma , \wghtw }   \Bigg].
\end{align*}
This finishes the proof. \end{proof}

Similarly, we have the following result.
\begin{proposition} \label{86}
Suppose $\kp _{1 , j} ^{2 } (\pvr) \wghtv _1 (\pvr )  \leq \ck _{2} \wghtv _2 (\pvr )$ for   a.e. $\pvr  \in \hat U _{1,j}$.    Assume that   $\fc _2, f \in L^{q_2} (\hat U _{1,j} \cap \Omega  ; \wghtv _2)$  for some  $q _2 >  \xv / (\xv -2)  $. Suppose  
\begin{equation*}
\begin{aligned}
 \ops [u,v_{h} ] + \int  _{\Omega } \fc _2 u v_{h}   \,  \dwv _{2}   
 \leq \int  _{\Omega }    f v_{ h } \, \dwv _{2}  ,  
\end{aligned}
\end{equation*}
for all $h\geq 0 $,   $v _{ h }  = u_{ h }  \zeta ^2 $ and $\zeta \in C^{\infty} _0 ( \hat U_{1,j} )$. Then
\begin{align*}
& \stackbin[ U _{1,j} \cap \Omega ]{}{\operatorname{ess}\sup}    u^{+}   \leq  C\left( 1 +  \widehat{\fc}_2 ^{\frac{\qxp _2 \xv }{\qxp _2 (\xv -2) - \xv}}      \right)\Vert u^+ \Vert _{2 , \hat U _{1,j} \cap \Omega   , \wghtv _2}  +  C  \Vert f \Vert _{ q_2 ,  \hat U _{1,j}  \cap \Omega , \wghtv _2}  (1+\widehat{\fc}_2 )^{-1/2} ,
\end{align*}
where  $C=C (  n , \mathcal C_{\xv} , \xv , \qxp _2 , \ck _2  ) >0$.
\end{proposition}

\begin{proposition} \label{53}
 Suppose $\kp _{0 , i} ^{2 } (\pvr) \wghtv _1 (\pvr )  \leq \ck _{2} \wghtv _2 (\pvr )$ for   a.e. $\pvr  \in \hat U _{0,i}$.    Assume that   $\fc _2, f \in L^{q_2} ( \hat U_{0,i}    ; \wghtv _2)$,   for some  $q _2 >  \xv / (\xv -2) $. Suppose  
\begin{equation*}
\begin{aligned}
 \ops [u,v_{h} ] + \int  _{\Omega } \fc _2 u v_{h}  \,  \dwv _{2}    
 \leq \int  _{\Omega}    f v_{ h } \, \dwv _{2}  ,
\end{aligned}
\end{equation*}
for all $h\geq 0 $,   $v _{ h }  = u_{ h }  \zeta ^2 $ and $\zeta \in C^{\infty} _0 ( \hat U_{0,i} )$. Then
\begin{align*}
 &\stackbin[ U_{0,i}]{}{\operatorname{ess}\sup}  u^{+} 
   \leq  C \left( 1+ \|\fc _2 \|_{\qxp _2 , \hat U_{0,i} , \wghtv _2} ^{\frac{\qxp _2 \xv }{\qxp _2 (\xv -2) - \xv}} \right) \Vert u^+ \Vert _{2 , \hat U _{0,i} , \wghtv _2}  +  C  \Vert f \Vert _{ q_2 ,  \hat U _{0,i} , \wghtv _2}   \left(1+\|\fc _2 \|_{\qxp _2 , \hat U_{0,i} , \wghtv _2}  \right)^{-1/2},
\end{align*}
where  $C=C (  n , \mathcal C_{\xv} ,    \xv ,  \qxp _2  , \ck _2)  >0$.
\end{proposition}
 
 
\subsection{Proof of Theorem \ref{59}}

{\bf Step 1.} {For every $u\in\womd$, the functional
$$
v\longmapsto\int_\Omega\iwe uv\,\dwv_2
$$
is bounded on $\womd$. Indeed, by H\"older's inequality and \eqref{62},
$$
\left|\int_\Omega\iwe uv\,\dwv_2\right|
\leq
\|\iwe\|_{\xv/(\xv-2),\Omega,\wghtv_2}
\|u\|_{\xv,\Omega,\wghtv_2}
\|v\|_{\xv,\Omega,\wghtv_2}
\leq
\mathcal C_{\xv}^2
\|\iwe\|_{\xv/(\xv-2),\Omega,\wghtv_2}
\|u\|_{1,2,\Omega,\wghtv_0,\wghtv_1}
\|v\|_{1,2,\Omega,\wghtv_0,\wghtv_1}.
$$
Therefore, by the Lax--Milgram theorem, there exists a unique $Au\in\womd$ such that
$$
\ops[Au,v]
=
( \iwe u , v)_{\wghtv_2} 
\quad\forall v\in\womd.
$$
where 
$$
(w , z)_{\wghtv_2} := \int_\Omega wz\,\dwv_2.
$$

The operator $A:\womd\rightarrow\womd$ is bounded and symmetric with respect to the inner product $\ops[\cdot,\cdot]$. Indeed,
$$
\ops[Au,v]=(\iwe u,v)_{\wghtv_2}=(\iwe v,u)_{\wghtv_2}=\ops[u,Av]
\quad\forall u,v\in\womd.
$$
Since $A$ is bounded and defined on all of $\womd$, it is self-adjoint.}

\medskip

{\bf Step 2.} {\it $A$ is compact and has infinitely many positive and negative eigenvalues.}

\medskip

Suppose first that $\spt\iwe\subset\Omega_m$ for some $m\in\mathbb N$. Let $(u_i)$ be bounded in $\womd$. Then
$$
\begin{aligned}
\|Au_i-Au_j&\|_{1,2,\Omega,\wghtv_0,\wghtv_1}^2\\[3pt]
=&\,\ops[A(u_i-u_j),A(u_i-u_j)]\\[5pt]
\leq&\,\|\iwe\|_{L^\infty(\Omega_m)}
\|u_i-u_j\|_{\xv/(\xv-1),\Omega_m,\wghtv_2}
\|Au_i-Au_j\|_{\xv,\Omega_m,\wghtv_2}.
\end{aligned}
$$
From \eqref{62},
$$
\|Au_i-Au_j\|_{1,2,\Omega,\wghtv_0,\wghtv_1}
\leq
\mathcal C_{\xv}\|\iwe\|_{L^\infty(\Omega)}
\|u_i-u_j\|_{\xv/(\xv-1),\Omega_m,\wghtv_2}.
$$
Since the weights are bounded above and below by positive constants on compact subsets of $M$ and
$$
\frac{\xv}{\xv-1}<2,
$$
the embedding
$$
W^{1,2}(\Omega_m;\wghtv_0,\wghtv_1)
\longrightarrow
L^{\xv/(\xv-1)}(\Omega_m;\wghtv_2)
$$
is compact. Hence, $A$ is compact in this case.

In general, set
$$
\iwe_m(\pvr):=\left\{\begin{aligned}
&\iwe(\pvr) &&\text{if }\pvr\in\Omega_m,\\
&0 &&\text{if }\pvr\in\Omega\setminus\Omega_m,
\end{aligned}\right.
$$
and define $A_m:\womd\rightarrow\womd$ by
$$
(\iwe_m u,v)_{\wghtv_2}=\ops[A_mu,v]
\quad\forall u,v\in\womd.
$$
Then
$$
\begin{aligned}
\|A_mu-Au&\|_{1,2,\Omega,\wghtv_0,\wghtv_1}^2\\[3pt]
\leq&\,
\|\iwe_m-\iwe\|_{\xv/(\xv-2),\Omega,\wghtv_2}
\|u\|_{\xv,\Omega,\wghtv_2}
\|A_mu-Au\|_{\xv,\Omega,\wghtv_2}.
\end{aligned}
$$
From \eqref{62},
$$
\|A_mu-Au\|_{1,2,\Omega,\wghtv_0,\wghtv_1}
\leq
C\|\iwe_m-\iwe\|_{\xv/(\xv-2),\Omega,\wghtv_2}
\|u\|_{\xv,\Omega,\wghtv_2}.
$$
Thus, $A_m\rightarrow A$ in operator norm. Since every $A_m$ is compact, $A$ is compact.

We now prove that $A$ has infinitely many eigenvalues of each sign. Choose nonempty open sets
$$
O_+\Subset\{\pvr\in\Omega\mid\iwe(\pvr)>0\},
\qquad
O_-\Subset\{\pvr\in\Omega\mid\iwe(\pvr)<0\}.
$$

For every $k\in\mathbb N$, choose linearly independent functions
$$
\zeta _1^+,\ldots,\zeta _k^+\in C_0^\infty(O_+).
$$
For every nonzero $\zeta \in\operatorname{span}\{\zeta _1^+,\ldots,\zeta _k^+\}$,
$$
\ops[A\zeta ,\zeta ]
=
\int_\Omega\iwe\zeta ^2\,\dwv_2>0.
$$

Therefore, by \cite[Theorem 4.15, Section 4.5]{zbMATH06377290},
$$
\dim\operatorname{Ran}P_A((0,\infty))\geq k.
$$
Since $k$ is arbitrary, the range of the positive spectral
projection $P_A((0,\infty))$ is infinite-dimensional. The same argument in $O_-$ gives
$$
\dim\operatorname{Ran}P_A((-\infty,0))=\infty.
$$
By \cite[Theorem 6.6, Section 6.2]{zbMATH06377290}, the nonzero spectrum of the compact self-adjoint operator $A$ consists of eigenvalues of finite multiplicity whose only possible accumulation point is $0$.

Hence, \eqref{61} has infinitely many eigenvalues
$$
\cdots\leq\lambda_2^-\leq\lambda_1^-<0<\lambda_1^+\leq\lambda_2^+\leq\cdots,
$$
with $\lambda_k^-\rightarrow-\infty$ and $\lambda_k^+\rightarrow+\infty$.

Moreover, as a consequence of Propositions \ref{86} and \ref{53}, provided that the first inequality in \ref{47} is also satisfied, every eigenfunction belongs to \(L^\infty_{\loc}(\overline\Omega)\).

\medskip

{\bf Step 3.} {\it If \ref{47} holds, then $u\in L^\infty(\Omega)$. Moreover, if $u\in W^{1,2}_0(\overline\Omega;\wghtv_0,\wghtv_1)$, then}
\begin{gather}\label{88}
\lim_{m\rightarrow\infty}
\stackbin[\Omega^m]{}{\operatorname{ess}\sup}\,|u|=0.
\end{gather}

Let $\qxp_2>\xv/(\xv-2)$. Applying Propositions \ref{86} and \ref{53} to $u$ and $-u$, with $\fc_2=-\lambda\iwe$ and $f=0$, there exists
$$
C_1=C_1(n,\mathcal C_{\xv},\xv,\qxp_2,\ck_2)>0
$$
such that
\begin{gather*}
\stackbin[U_{1,j}\cap\Omega]{}{\operatorname{ess}\sup}\,|u|
\leq
C_1\left(
1+\|\lambda\iwe\|_{\qxp_2,\hat U_{1,j}\cap\Omega,\wghtv_2}^{\frac{\qxp_2\xv}{\qxp_2(\xv-2)-\xv}}
\right)
\|u\|_{2,\hat U_{1,j}\cap\Omega,\wghtv_2},\\[5pt]
\stackbin[U_{0,i}]{}{\operatorname{ess}\sup}\,|u|
\leq
C_1\left(
1+\|\lambda\iwe\|_{\qxp_2,\hat U_{0,i},\wghtv_2}^{\frac{\qxp_2\xv}{\qxp_2(\xv-2)-\xv}}
\right)
\|u\|_{2,\hat U_{0,i},\wghtv_2}
\end{gather*}
for all $i,j\in\mathbb N$.

From \ref{47},
$$
\|\lambda\iwe\|_{\qxp_2,\hat U_{k,i}\cap\Omega,\wghtv_2}
\leq
|\lambda|\|\iwe\|_{L^\infty(\Omega)}\ck_3^{1/\qxp_2}
$$
and
$$
\|u\|_{2,\hat U_{k,i}\cap\Omega,\wghtv_2}
\leq
\ck_3^{\frac{\xv-2}{2(\xv)}}
\|u\|_{\xv,\hat U_{k,i}\cap\Omega,\wghtv_2}
\leq
\ck_3^{\frac{\xv-2}{2(\xv)}}\mathcal C_{\xv}
\|u\|_{1,2,\Omega,\wghtv_0,\wghtv_1}.
$$
The preceding estimates are uniform in $i$ and $j$. Since the sets $U_{k,i}$ cover $\overline\Omega$, we obtain $u\in L^\infty(\Omega)$.

Let $\epsilon>0$. Since $u\in W^{1,2}_0(\overline\Omega;\wghtv_0,\wghtv_1)$, there exists $\zeta\in C_0^\infty(M)$ such that
$$
\|u-\zeta\|_{1,2,\Omega,\wghtv_0,\wghtv_1}<\epsilon.
$$
Let $m_0\in\mathbb N$ be such that $\spt\zeta\subset D_{m_0}$.

By the local finiteness of $\{\hat U_{k,i}\}$, only finitely many sets $\hat U_{k,i}$ meet $\overline{D_{m_0}}$. Since every corresponding $U_{k,i}$ is relatively compact, there exists $m_1\geq m_0$ such that, whenever $U_{k,i}\cap D^{m_1}\neq\varnothing$, one has
$$
\hat U_{k,i}\cap\spt\zeta=\varnothing.
$$
For these indices, $\zeta=0$ in $\hat U_{k,i}$ and, from \eqref{62},
$$
\|u\|_{\xv,\hat U_{k,i}\cap\Omega,\wghtv_2}
=
\|u-\zeta\|_{\xv,\hat U_{k,i}\cap\Omega,\wghtv_2}
\leq
\mathcal C_{\xv}\epsilon.
$$

\begin{align*}
 \stackbin[ U _{1,j} \cap \Omega ]{}{\operatorname{ess}\sup}  \,  |u| 
   \leq  & \displaystyle \, C_1 \left[  1+  \left(\|\lambda \iwe\|_{L^\infty (\Omega)} ^{\qxp _2} \ck _3 \right) ^{ \frac{\xv}{\qxp _2 (\xv -2) - \xv}} \right] \ck _3 ^{\frac{\xv -2}{\xv}}\Vert u \Vert _{\xv , \hat U_{1,j} \cap  \Omega   , \wghtv _2}  \\[3pt]
 \leq & \displaystyle \,  C_1 \left[ 1+  \left(\|\lambda \iwe\|_{L^\infty (\Omega)} ^{\qxp _2} \ck _3 \right) ^{ \frac{\xv }{\qxp _2 (\xv -2) - \xv}}    \right] \\[3pt]
& \displaystyle \, \cdot  \ck _3 ^{\frac{\xv -2}{\xv}} \left( \Vert u - \zeta \Vert _{\xv ,  \hat U_ {1,j}  \cap \Omega   , \wghtv _2} + \Vert  \zeta \Vert _{\xv ,  \hat U_ {1,j}  \cap \Omega   , \wghtv _2} \right)    \\[5pt]
  \leq & \displaystyle \,  C_1 \left[ 1+  \left(\|\lambda \iwe\|_{L^\infty (\Omega)} ^{\qxp _2} \ck _3 \right) ^{ \frac{\xv }{\qxp _2 (\xv -2) - \xv}}   \right] \mathcal{C} _{\xv}\ck _3 ^{\frac{\xv -2}{\xv}}  \epsilon 
\end{align*}
and 
\begin{align*}
&\stackbin[ U_{0,i}]{}{\operatorname{ess}\sup} \, |u|   \leq  C_1 \left[1+ \left(\|\lambda\iwe\|_{L^\infty (\Omega)} ^{\qxp _2} \ck _3 \right) ^{ \frac{\xv }{\qxp _2 (\xv -2) - \xv}}   \right]\mathcal{C} _{\xv}\ck _3 ^{\frac{\xv -2}{\xv}}       \epsilon.
\end{align*}

Since the right-hand side is independent of $(k,i)$ and the sets $U_{k,i}$ cover $\overline\Omega$, we obtain
$$
\stackbin[\Omega^m]{}{\operatorname{ess}\sup}\,|u|
\leq
C\epsilon
\quad\text{for every }m\geq m_1.
$$
As $\epsilon>0$ is arbitrary, \eqref{88} follows.

This concludes the proof of Theorem \ref{59}.


\appendix

\section{Proof of Lemma \ref{alcontb0}}
 Similarly to Lemma \ref{6}, we have the following result.
\begin{lemma} \label{44}
Suppose $0<r<R\leq \hat r _{1,j}$ and fix $j\in \mathbb{N}$. There exists a smooth function $\zeta  : \hat U  _{1,j} \rightarrow [0,1]$ such that  
$$
\zeta  = 1 \text {  in } \psi _{1,j} (B(0,r)),   \quad \spt \zeta  \subset \psi _{1,j} (B(0,R)),
$$
 and 
 $$
 |\nabla \zeta  | _g \leq \cen _1(n)\bnorm d (\psi _{1,j}^{-1}) \bnorm  (R-r)^{-1} \quad \text { in } \quad \hat  U _{1,j}.
 $$ 
\end{lemma}
 
We follow the proof in \cite[Theorem 4.1]{hanlin2011ellipticpartial}; see also \cite[Lemma A.2]{apaza2024yamabecorner}.
\begin{cl} \label{accontb1}
We have
\begin{equation*}
\begin{aligned}
& \int   _{\hat U _{1,j} \cap \Omega }    |\nabla (u_{h_{2}} \zeta  )|^2 \, \dwv _{1}  \\[3pt]
& \quad \ \leq C_1( n) \left[ \frac{(\hat r _{1,j} - r_{1,j} )^2 } {(R-r) ^{2}} \ck _2\int _{\mathscr{U}(h_{2},R)} u_{ h_{2} }^2 \, \dwv _{2}+ \int  _{\hat U _{1,j} \cap \Omega }    g (\nabla u , \nabla v_{ h_{2} }) \, \dwv _{1}  \right].
\end{aligned}
\end{equation*}
\end{cl} 
\begin{proof}
We have 
\begin{align*}
\int _{\hat U _{1,j} \cap \Omega } g(\nabla u , \nabla v_{h_{2}}) \, \dwv _{1} = & \displaystyle \, \int _{\hat U _{1,j} \cap \Omega } \left( \zeta  ^2 g(\nabla u , \nabla u_{h_{2}} ) + 2u_{ h_{2} } \zeta  g(\nabla u , \nabla \zeta   )  \right) \, \dwv _{1} \\[3pt]
 \geq & \displaystyle \,   \int _{\hat U _{1,j} \cap \Omega } \zeta  ^2  | \nabla u_{h_{2}} |^2 \, \dwv _{1}   - \int _{\hat U _{1,j} \cap \Omega } 2 \zeta   |\nabla u _{h_{2}}  | \cdot |\nabla \zeta   |  \cdot |u_{h_{2}} |  \, \dwv _{1} \\[5pt]
 \geq & \displaystyle \,  \frac{1}{2}\int _{\hat U _{1,j} \cap \Omega } \zeta  ^2  | \nabla u_{h_{2}} |^2 \, \dwv _{1}   - 2\int _{\hat U _{1,j} \cap \Omega }   |\nabla \zeta   |^2  u_{h_{2}}  ^2 \, \dwv _{1} .
\end{align*}
From Lemma \ref{44} and condition \ref{47}, we conclude the claim.\end{proof}

The inequality \eqref{a168} and  Claim \ref{accontb1} imply 
\begin{equation} \label{a169}
\begin{aligned}
\min \{1, & C_1\}  \ops [u_{h_{2}}  \zeta , u_{h_{2}}  \zeta] \\[3pt]
 \leq & \displaystyle \,  \int  _ {\hat U _{1,j} \cap \Omega }  |\nabla (u_{h_{2}}  \zeta  )|^2 \, \dwv _{1}   + C_1 \int  _{\hat U _{1,j} \cap \Omega }  u v_{h_{2}} \,  \dwv _{0}  \\[3pt]
 \leq  & \displaystyle \,    C_1 \left[ \frac{ (\hat r _{1,j} - r_{1,j})^2}{(R-r) ^{2}} \ck _{2} \int _{\mathscr{U}(h_{2},R)} u_{h_{2}}  ^2 \, { \dwv _{2}} \right. - \int  _{\hat U _{1,j} \cap \Omega } \fc  _2 u v_{h_{2}}   \, \dwv _{2}    \\[5pt]
& \displaystyle \,  \left.      - \int  _{ \hat U _{1,j} \cap \Gamma} \fc _3 u v_{h_{2}}  \, \dww   + \int  _{\hat U _{1,j} \cap \Omega }    f v_{h_{2}}  \, \dwv _{2} + \int_{\hat U _{1,j} \cap \Gamma }  f_1 v _{h_{2}} \, \dww  \right]  .
\end{aligned} 
\end{equation}

Next, we will estimate the terms on the right-hand side of \eqref{a169}.

\begin{cl}\label{accontb2}
The following estimates are valid:

\medskip

\noindent $(i)$ One has
\begin{align*}
-\int _{\hat U _{1,j} \cap \Gamma }  \fc _3  u & (u_{h_{2}}  \zeta ^2 ) \,  \dww    \\[3pt] 
 \leq & \displaystyle \,  2 \mathcal C ^2 _{\xw} \Vert \fc _3 \Vert _{q_3 ,\hat U _{1,j} \cap \Gamma , \wghtw}  \wghtw  ( \{u_{h_{2}}  \zeta \neq 0\} \cap \hat U _{1,j} \cap \Gamma ) ^ {\frac{\xw - 2}{\xw} - \frac{1}{q_3}}  \ops [ u_{h_{2}}  \zeta , u_{h_{2}}  \zeta ] \\[5pt]
& \displaystyle \, +  h_{2}^2 \Vert \fc _3 \Vert _{q_3 , \hat U _{1,j} \cap \Gamma , \wghtw } \wghtw  (\{ u_{h_{2}}  \zeta \neq 0\} \cap \hat U _{1,j} \cap \Gamma )^{1-\frac{1}{q_3}}. 
\end{align*}

\noindent  $(ii)$ For all  $\delta >0$,
\begin{align*}
&\int _{\hat U _{1,j} \cap \Gamma }  f_1 u_{h_{2}}  \zeta ^2  \dww\\[3pt]  
& \quad \ \leq \mathcal C_{\xw} \left( \delta ^{-1}  \wghtw  (\{ u_{h_{2}}  \zeta \neq 0\}\cap \hat U _{1,j} \cap \Gamma ) ^ {  \frac{2(\xw -1)}{\xw} - \frac{2}{q_3}} \Vert f_1 \Vert _{ q_3 , \hat U _{1,j} \cap \Gamma , \wghtw } ^2  + \delta \ops [ u_{h_{2}}  \zeta  , u_{h_{2}}  \zeta ] \right).
\end{align*} 

\medskip

\noindent  $(iii)$ One has
\begin{align*}
-\int _{\hat U _{1,j} \cap \Omega }  \fc _2 u & ( u_{h_{2}}  \zeta ^2 ) \, \dwv _{2} \\[3pt]
 \leq & \displaystyle \,  2 \mathcal C_{\xv} ^2\Vert \fc _2 \Vert _{q_2 , \hat U _{1,j} \cap \Omega  , \wghtv _2}  \wghtv _2( \{u_{h_{2}}  \zeta \neq 0\} ) ^ {\frac{\xv -2}{\xv} - \frac{1}{q_2}} \ops [ u_{h_{2}}  \zeta , u_{h_{2}}  \zeta] \\[5pt]
& \displaystyle \, +  h_{2}^2 \Vert \fc _2 \Vert _{q_2 ,  \hat U _{1,j} \cap \Omega   , \wghtv _2} \wghtv _2 ( \{u_{h_{2}}  \zeta \neq 0\} )^{1-\frac{1}{q_2}}.
\end{align*}

\noindent  $(iv)$ For all $\delta >0$,
\begin{align*}
\int _{\hat U _{1,j} \cap \Omega }  f u_{h_{2}}  & \zeta ^2 \, \dwv _{2} \\[3pt]
 \leq  & \displaystyle \, \mathcal C_{\xv} \left( \delta ^{-1} \wghtv _2 ( \{u_{h_{2}}  \zeta \neq 0\}) ^ {\frac{2(\xv -1)}{\xv} - \frac{2}{q _2}} \Vert f \Vert _{q_2 ,  \hat U _{1,j} \cap \Omega   , \wghtv _2 } ^2  + \delta  \ops [u_{h_{2}}  \zeta , u_{h_{2}}  \zeta ] \right).
\end{align*}

\end{cl}
\begin{proof} We start with the following inequalities, which will be used later. Hölder's inequality and \eqref{63} imply
\begin{equation} \label{aecontb1}
\begin{aligned}
 \int _{\hat U _{1,j} \cap \Gamma }  (u_{h_{2}}  \zeta )^2 \, \dww \leq & \displaystyle \,  \wghtw  (\{ u_{h_{2}}  \zeta \neq 0 \} \cap \hat U _{1,j} \cap \Gamma )^{1-\frac{2}{\xw}}\left( \int _{\hat U _{1,j} \cap \Gamma }  | u_{h_{2}}  \zeta |^{\xw} \dww \right)^{\frac{2}{\xw}} \\[3pt]
  \leq  & \displaystyle \,  \mathcal{C} ^2 _{\xw} \wghtw  ( \{ u_{h_{2}}  \zeta \neq 0 \}\cap \hat U _{1,j} \cap \Gamma )^{1-\frac{2}{\xw}} \ops [ u_{h_{2}}  \zeta , u_{h_{2}}  \zeta].
 \end{aligned}
\end{equation}
On the other hand, by  \eqref{62},
\begin{equation} \label{aecontb2}
\begin{aligned}
 \int _{\hat U _{1,j} \cap \Omega } (u_{h_{2}}  \zeta )^2 \, \dwv _{2} \leq & \displaystyle \,  \wghtv _{2}  (\{ u_{h_{2}}  \zeta \neq 0 \}  )^{1-\frac{2}{\xv}}\left( \int _{\hat U _{1,j} \cap \Omega } | u_{h_{2}}  \zeta |^{\xv} \, \dwv _{2} \right)^{\frac{2}{\xv}} \\[3pt]
  \leq & \displaystyle \,  \mathcal{C} ^2 _{\xv} \wghtv _{2}  ( \{ u_{h_{2}}  \zeta \neq 0 \} )^{1-\frac{2}{\xv}} \ops [ u_{h_{2}}  \zeta , u_{h_{2}}  \zeta].
 \end{aligned}
\end{equation}

\medskip

\noindent $(i)$ \, Since $2/\xw + 1/q_3 <1$,
\begin{align*}
-\int _{\hat U _{1,j} \cap \Gamma }  \fc _3  u (u_{h_{2}}  \zeta ^2 )  \,  \dww   = & \displaystyle \,  -\int _{\{ u_{h_{2}}  \zeta  \neq 0\} \cap \hat U _{1,j} \cap \Gamma }  \fc _3 \left(  u_{h_{2}}  ^2 + h_{2}  u_{h_{2}}  \right) \zeta ^2   \dww   \\[3pt]
 \leq & \displaystyle \, 2 \int _{\{ u_{h_{2}}  \zeta  \neq 0 \} \cap \hat U _{1,j} \cap \Gamma }  |\fc _3|  u_{h_{2}}  ^2 \zeta ^2  \, \dww   + h_{2}^2    \int _{\{ u_{h_{2}}  \zeta  \neq 0 \} \cap \hat U _{1,j} \cap \Gamma }     |\fc _3|\zeta ^2  \, \dww   \\[3pt]
 \leq & \displaystyle \,  2\left( \int _{\hat U _{1,j} \cap \Gamma }  |\fc _3| ^{q_3}  \dww  \right) ^{\frac{1}{q_3}} \left( \int _{\{ u_{h_{2}}  \zeta \neq 0 \} \cap \hat U _{1,j} \cap \Gamma }   | u_{h_{2}}  \zeta |^ {\xw }  \dww    \right) ^{\frac{2}{\xw}}  \\[3pt]
& \displaystyle \, \cdot \wghtw  ( \{ u_{h_{2}}  \zeta \neq 0 \} \cap \hat U _{1,j} \cap \Gamma  ) ^ {1-\frac{2}{\xw} - \frac{1}{q_3}} \\[3pt]
& \displaystyle \,  + h_{2}^2 \Vert \fc _3 \Vert _{q_3 , \hat U _{1,j} \cap \Gamma , \wghtw}  \wghtw  (  \{ u_{h_{2}}  \zeta \neq 0 \} \cap \hat U _{1,j} \cap \Gamma )^{1-\frac{1}{q_3}}\\[7pt]
 \leq   & \displaystyle \, 2 \mathcal C ^2 _{\xw} \Vert \fc _3 \Vert _{ q_3 , \hat U _{1,j} \cap \Gamma , \wghtw}  \wghtw  (  \{ u_{h_{2}}  \zeta \neq 0 \} \cap \hat U _{1,j} \cap \Gamma ) ^ {\frac{\xw -2}{\xw} - \frac{1}{q_3}} \ops [ u_{h_{2}}  \zeta , u_{h_{2}}  \zeta] \\[3pt]
& \displaystyle \,  + h_{2}^2 \Vert \fc _3 \Vert _{q_3 , \hat U _{1,j} \cap \Gamma , \wghtw  } \wghtw  (  \{ u_{h_{2}}  \zeta \neq 0 \} \cap \hat U _{1,j} \cap \Gamma )^{1-\frac{1}{q_3}}, 
\end{align*}
by \eqref{aecontb1}. This concludes the proof of $(i)$.\\

\noindent $(ii)$ \ Since      $1/q_3 + 1/\xw < 1$  and  $0\leq \zeta \leq 1$,
\begin{align*}
& \int _{\hat U _{1,j} \cap \Gamma }  f_1 u_{h_{2}}  \zeta ^2  \dww    \\[3pt]
& \quad \  \leq   \left( \int _{\hat U _{1,j} \cap \Gamma }  |f_1| ^{q_3}  \dww   \right) ^{\frac{1}{q_3}} \left( \int _{\hat U _{1,j} \cap \Gamma }   | u_{h_{2}}  \zeta |^ {\xw}  \dww    \right) ^{ \frac{1}{\xw} }  \wghtw  (\{ u_{h_{2}}  \zeta \neq 0\}\cap \hat U _{1,j} \cap \Gamma ) ^ {1- \frac{1}{q_3} - \frac{1}{\xw}}.
\end{align*}
By \eqref{aecontb1},
\begin{align*}
&\int _{\hat U _{1,j} \cap \Gamma }  |f_1| u_{h_{2}}  \zeta ^2   \dww    \\[3pt]
& \quad \ \leq  \mathcal C_{\xw} \Vert f_1 \Vert _{q_3 , \hat U _{1,j} \cap \Gamma , \wghtw }  \sqrt{\ops [ u_{h_{2}}  \zeta , u_{h_{2}}  \zeta ]}   \wghtw  (\{ u_{h_{2}}  \zeta \neq 0\}\cap \hat U _{1,j} \cap \Gamma ) ^ {\frac{\xw -1 }{\xw} - \frac{1}{q_3}} \\[5pt]
&\quad \ \leq  \mathcal C_{\xw} \delta ^{-1} \wghtw  ( \{ u_{h_{2}}  \zeta \neq 0\}\cap \hat U _{1,j} \cap \Gamma ) ^ {\frac{2(\xw -1)}{\xw} - \frac{2}{q_3}}  \Vert f_1 \Vert _{q_3 , \hat U _{1,j} \cap \Gamma , \wghtw} ^2  +  \mathcal C_{\xw} \delta \ops [ u_{h_{2}}  \zeta , u_{h_{2}}  \zeta ].
\end{align*}
This proves $(ii)$. The proofs of $(iii)$ and $(iv)$ follow the same lines as in $(i)$ and $(ii)$. \end{proof}

Let us observe that  
\begin{gather*}
1- \frac{ 1}{q_2} \leq \frac{2(\xv -1) }{ \xv} - \frac{2}{q_2}, \quad   1- \frac{1}{q_3} \leq \frac{2(\xw -1) }{ \xw} - \frac{2}{q_3},\\[3pt]  
\{u_{h_{2}}  \zeta  \neq 0\} \subset \mathscr{U}(h_{2},R),  \quad \{u_{h_{2}}  \zeta  \neq 0\} \cap \hat U _{1,j} \cap \Gamma \subset \mathscr{U} _{\Gamma }  (h_{2},R),\\[3pt]
  \wghtv _{2} (\mathscr{U} (h_{2},R) ) \leq h_{2}^{-2} \int _{\mathscr{U} (h_{2},R)} (u^+)^2 \, \dwv _{2} , \quad   \text {  and  } \quad \wghtw ( \mathscr{U} _{\Gamma }  (h_{2},R) ) \leq h_{2}^{-2} \int _{\mathscr{U} _{\Gamma } (h_{2},R)} (u^+)^2  \, \dww  .
  \end{gather*} 

Hence, by \eqref{a169} and Claim \ref{accontb2}, there exists a constant $N = N (  n , \mathcal C_{\xv} , \mathcal C_{\xw} , \xv , \xw , \qxp _2 , \qxp _3 ) >0 $ such that if  
$$
h_{2}\geq N \max \left\{   \|\fc _2 \|_{\qxp _2 , \hat U_{1,j} \cap \Omega , \wghtv _2} ^{\frac{\qxp _2 \xv }{\qxp _2 (\xv -2) - \xv}}   , \|\fc _3 \|_{\qxp _3 , \Gamma \cap \hat U_{1,j}, \wghtw} ^{\frac{\qxp _3 \xw }{\qxp _3 (\xw -2) - \xw}} , 1\right\}  \max \left\{ \Vert u^+ \Vert _{2 , \hat U _{1,j} \cap \Omega  , \wghtv _2 }  ,   \Vert  u^+ \Vert _{2 , \hat U _{1,j} \cap \Gamma , \wghtw }   \right\},
$$
then
\begin{equation}
\wghtv _2(\mathscr{U}(h_{2},R))<1 , \quad   \wghtw  ( \mathscr{U} _{\Gamma }  (h_{2},R)) \ <1 \label{aecontb5}
\end{equation}
and
\begin{equation} \label{aecontb4}
\begin{split} 
C_{2} \ops [ u_{h_{2}}  \zeta , &  u_{h_{2}}  \zeta] \\[3pt]
 \leq  & \displaystyle \,  \frac{ (\hat r _{1,j} - r_{1,j})^2}{(R-r) ^{2}} \ck _{2} \int _{\mathscr{U}(h_{2},R)} u_{h_{2}}  ^2   \, \dwv _{2}  +    \left(  \Vert f \Vert _{q_2 , \hat U _{1,j} \cap \Omega  , \wghtv _2 } ^2 + h_{2}^2 \widehat{\fc}_2 \right)  \wghtv _2 ( \{u_{h_{2}}  \zeta \neq 0\})^{1-\frac{1}{q_2}}  \\[5pt]
 & \displaystyle \, + \left(  \Vert f_1 \Vert _{q_3 ,\hat U _{1,j} \cap \Gamma , \wghtw} ^2 + h_{2}^2 \widehat{\fc}_3 \right) \wghtw (\{ u_{h_{2}}  \zeta \neq 0\}\cap \hat U _{1,j} \cap \Gamma)^{1-\frac{1}{q_3}} ,
\end{split}
\end{equation}
where $C_{2}=C_{2}(  n , \mathcal C_{\xv} , \mathcal C_{\xw}  )>0$, and
$$
\widehat{\fc}_2 :=\Vert \fc _2 \Vert _{q_2 , \hat U _{1,j} \cap \Omega  , \wghtv _2 } ,\qquad \widehat{\fc}_3 :=\Vert \fc _3 \Vert _{q_3 , \hat U _{1,j} \cap \Gamma  , \wghtw } .
$$

From   \eqref{aecontb2} and   \eqref{aecontb4}, we have
\begin{equation}\label{aecontb6}
\begin{split} 
C_{3} \int   _{\hat U _{1,j} \cap \Omega } &    (u_{h_{2}}  \zeta )^2 \, \dwv _{2}   \\[3pt]
 \leq & \displaystyle \,    \frac{ (\hat r _{1,j} - r_{1,j})^2}{(R-r) ^{2}} \wghtv _2 ( \{u_{h_{2}}  \zeta \neq 0\})^{1- \frac{2}{\xv}}  \left( \int _{\mathscr{U}(h_{2},R)} u_{h_{2}}  ^2 \, \dwv _{2}    + \int _{\mathscr{U} _{\Gamma } (h_{2},R)}  u_{h_{2}}  ^2 \, \dww   \right)   \\[3pt]
 & \displaystyle \, + \left[  \Vert f \Vert _{q_2 , \hat U _{1,j} \cap \Omega  , \wghtv _2 } + \Vert f_1 \Vert _{q_3 , \hat U _{1,j} \cap \Gamma , \wghtw } + h_{2} (\widehat{\fc}_2+\widehat{\fc}_3 )^{1/2} \right]^2 \\[3pt]
& \displaystyle \,  \cdot   \left( \wghtv _2 ( \{u_{h_{2}}  \zeta \neq 0\})^{1- \frac{2}{\xv} + 1 -\frac{1}{q_2} }  + \wghtv _2 ( \{u_{h_{2}}  \zeta \neq 0\})^{1- \frac{2}{\xv}}  \wghtw  ( \{ u_{h_{2}}  \zeta \neq 0\}\cap \hat U _{1,j} \cap \Gamma )^{1-\frac{1}{q_3}} \right) .
\end{split}
\end{equation}
and, by \eqref{aecontb1},
\begin{equation}\label{aecontb7}
\begin{split}
C_{4} \int _{\hat U _{1,j} \cap \Gamma } & ( u_{h_{2}}  \zeta)^2   \dww   \\[3pt]
 \leq & \displaystyle \,    \frac{ (\hat r _{1,j} - r_{1,j})^2}{(R-r) ^{2}}    \wghtw  ( \{ u_{h_{2}}  \zeta \neq 0\} \cap \hat U _{1,j} \cap \Gamma )^{1 - \frac{2}{\xw}} \left( \int _{\mathscr{U} (h_{2}, R )} u_{h_{2}}  ^2  \, \dwv _{2}   + \int _{\mathscr{U}_{\Gamma } (h_{2},R)} u_{h_{2}}  ^2  \dww   \right)       \\[3pt]
& \displaystyle \, + \left[  \Vert f \Vert _{q _2 , \hat U _{1,j} \cap \Omega   , \wghtv _2} +  \Vert f_1 \Vert _{q_3 , \hat U _{1,j} \cap \Gamma , \wghtw}  + h_{2} (\widehat{\fc}_2+\widehat{\fc}_3 )^{1/2}\right]^2 \\[3pt]
& \displaystyle \, \cdot  \left(\wghtw  (\{ u_{h_{2}}  \zeta \neq 0\} \cap \hat U _{1,j} \cap \Gamma )^{1 - \frac{2}{\xw}}  \wghtv _{2}  (\{u_{h_{2}}  \zeta \neq 0\} )^{1 -\frac{1}{q_2}}   \right. \\[3pt]
& \displaystyle \, \left. +   \wghtw  ( \{ u_{h_{2}}  \zeta \neq 0\} \cap \hat U _{1,j} \cap \Gamma )^{1 - \frac{2}{\xw}+ 1 -\frac{1}{q_3}} \right) ,
 \end{split}
\end{equation}
where $C_i = C_i (  n , \mathcal C_{\xv} , \mathcal C_{\xw} , \ck _{2} )>0$, $i=3, 4$.

Set
$$
\epsilon  := \min \left\{ 1-\frac{1}{q_2} -\frac{2}{\xw} , 1-\frac{1}{q_3} -\frac{2}{\xv} , 1-\frac{1}{q_2} -\frac{2}{\xv} , 1-\frac{1}{q_3} -\frac{2}{\xw}  \right\}>0
$$

 By Young's inequality,
\begin{equation}\label{a206}
\begin{aligned}
& \wghtv _2 (\{u_{h_{2}}  \zeta \neq 0\} )^{1-\frac{2}{\xv}}  \wghtw (\{ u_{h_{2}}  \zeta \neq 0\}  \cap \hat U _{1,j} \cap \Gamma ) ^{1 -\frac{1}{q_3}}\\[3pt] 
& \quad \ \leq  \frac{1}{C_{5 }}  \wghtv _2 ( \{u_{h_{2}}  \zeta \neq 0\} )^{1+\epsilon } +  \frac{C_{5 } -1}{C_{5 }} \wghtw ( \{ u_{h_{2}}  \zeta \neq 0\} \cap \hat U _{1,j} \cap \Gamma ) ^{\left(1-\frac{1}{q_3}\right)\frac{C_{5 }}{C_{5 }-1} }     ,
\end{aligned}
\end{equation}
\begin{equation} \label{a254}
\begin{aligned}
& \wghtv _2 ( \{u_{h_{2}}  \zeta \neq 0\} )^{1 -\frac{1}{q_2}}  \wghtw ( \{ u_{h_{2}}  \zeta \neq 0\} \cap \hat U _{1,j} \cap \Gamma )^{1- \frac{2}{\xw}}\\[3pt] 
& \quad \leq \frac{C_{6 }-1}{C_{6 }} \wghtv _2 ( \{ u_{h_{2}}  \zeta \neq 0\}  ) ^{\left(1-\frac{1}{q_2}\right)\frac{C_{6 }}{C_{6 }-1} }  + \frac{1}{C_{6 }}  \wghtw ( \{u_{h_{2}}  \zeta \neq 0\} \cap \hat U _{1,j} \cap \Gamma ) ^{1+\epsilon },
\end{aligned}
\end{equation}
where $C_{5 } (1- 2/\xv ) =1+\epsilon $ and $C_{6 }(1-2/\xw)= 1+\epsilon$. Observe also that 
$$
\left(1-\frac{1}{q_3}\right)\frac{C_{5 }}{C_{5 }-1} \geq 1+\epsilon  \ \text { and  } \ \left(1-\frac{1}{q_2}\right)\frac{C_{6 }}{C_{6 }-1}\geq 1+\epsilon .
$$

The inequalities \eqref{aecontb5}, \eqref{aecontb6} - \eqref{a254} imply
\begin{equation}\label{a151}
\begin{split}
C_{7} \Psi ^2 (h_{2} , r ) \leq & \, \displaystyle   \frac{ (\hat{r}_{1, j}-r_{1, j})^2}{(R-r)^2} \left(   \wghtv _2 (\{u_{h_{2}}  \zeta \neq 0 \})  +\wghtw ( \{u_{h_{2}}  \zeta \neq 0 \}\cap \hat U _{1,j} \cap \Gamma) \right)^{\epsilon }\Psi ( h_{2}, R ) ^2 \\[3pt]
& \, \displaystyle + \left[  \Vert f \Vert _{q_2 , \hat U _{1,j} \cap \Omega  , \wghtv _2 } +  \Vert f_1 \Vert _{q_3 , \hat U _{1,j} \cap \Gamma , \wghtw}  + h_{2} (\widehat{\fc}_2+\widehat{\fc}_3 )^{1/2} \right]^2 \\[3pt]
& \, \displaystyle \cdot \left( \wghtv _2 (\{u_{h_{2}}  \zeta \neq 0 \})  +\wghtw ( \{u_{h_{2}}  \zeta \neq 0 \}\cap \hat U _{1,j} \cap \Gamma) \right)^{1+\epsilon },
\end{split}
\end{equation}
where $C_{7}= C_{7} (  n , \mathcal C_{\xv} , \mathcal C_{\xw}  , \xv , \xw , \qxp _2 , \qxp _3 , \ck _{2}) >0$. Consider \eqref{a151} and the following claim, which is proved as in \cite{hanlin2011ellipticpartial}:
\begin{cl}
If $h_{2}>h_{1} $, then
\begin{gather*}
\int _{\mathscr{U}(h_{2},R)} u_{h_{2}}  ^2 \,  \dwv _2  \leq \int _{\mathscr{U} (h_{1} ,R)} u_{h_{1}}   ^2 \, \dwv _2  , \quad \int _{\mathscr{U} _{\Gamma } (h_{2},R)}  u_{h_{2}}  ^2   \dww  \leq \int _{\mathscr{U} _{\Gamma }  (h_{1} ,R)} u_{h_{1}}   ^2  \dww  ,\\[3pt]
 \wghtv _2 ( \{u_{h_{2}}  \zeta  \neq 0\} ) \leq  \frac{1}{(h_{2}-h_{1} )^2}\int _{\mathscr{U}(h_{1} ,R)} u_{h_{1}}   ^2\, \dwv _2, \\[3pt]
  \wghtw ( \{ u_{h_{2}}  \zeta \neq 0\} \cap \hat U _{1,j} \cap \Gamma ) \leq  \frac{1}{(h_{2}-h_{1} )^2}\int _{\mathscr{U}_{\Gamma } (h_{1} ,R)} u_{h_{1}}   ^2  \dww .
\end{gather*}
\end{cl}
\noindent Therefore,
\begin{align*}
& C_7 \Psi ^2 (h_{2},r) \\
& \quad \ \leq  \left\{ \frac{ (\hat r _{1,j} - r_{1,j})^2}{(R-r) ^{2}(h_{2}-h_{1} )^{2\epsilon }}  + \frac{  [  \Vert f \Vert _{q_2  ,  \hat U _{1,j} \cap \Omega  , \wghtv _2}  +  \Vert f_1 \Vert _{q_3 , \hat U _{1,j} \cap \Gamma , \wghtw} + h_{2} (\widehat{\fc}_2+\widehat{\fc}_3 )^{1/2}   ]^2}{(h_{2}-h_{1} )^{2(1+\epsilon )}}\right\} \Psi ^{2(1 + \epsilon  )} (h_{1} ,R).
\end{align*}
This proves Lemma \ref{alcontb0}.


\vspace{1cm}


\noindent {\bf Funding:}This work was supported by the Instituto Nacional de Ciência e Tecnologia de Matemática (INCTMat) through CNPq Grant No. 170245/2023-3, and by CNPq-Brazil under Grants Nos. 153232/2024-2 and 150680/2025-2.

\noindent {\bf Data Availability:} No data were used for the research described in the article.

\noindent {\bf Declarations}

\noindent {\bf Conflict of interest:} The author declares no conflict of interest.



 
 \bibliographystyle{abbrv}

    \bibliography{ref}

\end{document}